\documentclass[preprint]{imsart}

\usepackage{natbib}
\bibpunct{(}{)}{,}{a}{}{;}

\usepackage[T1]{fontenc}
\usepackage[utf8]{inputenc}
\usepackage[english]{babel}
\usepackage{amsmath,amsthm,amssymb,amsfonts,epic,latexsym,hyperref
}
\usepackage{mathabx}
\usepackage[dvips]{graphicx}
\usepackage[usenames]{color}

\usepackage{lipsum}
\usepackage{rotating}
\usepackage{dsfont}
\usepackage{geometry}
\usepackage{stmaryrd}
\usepackage{ifthen}
\usepackage[nottoc]{tocbibind}
\usepackage{verbatim}
\usepackage{tabularx}
\usepackage{framed}
\usepackage{longtable}
\usepackage{tabu}
\usepackage{mathtools}

\allowdisplaybreaks

\newcommand\un{\mathds{1}}
\newcommand\eps{\varepsilon}
\renewcommand\phi{\varphi}
\newcommand\pro[1]{\mathbb{P}\left(#1\right)}
\newcommand\esp[1]{\mathbb{E}\left[#1\right]}

\newcommand\uno[1]{\un_{\left\{#1\right\}}}

\newtheorem{thm}{Theorem}[section]
\newtheorem{prop}[thm]{Proposition}

\newtheorem{assu}{Assumption}
\newtheorem{lem}[thm]{Lemma}
\newtheorem{rem}[thm]{Remark}

\newtheorem{ex}{Example}

\newcommand\n{\mathbb{N}}

\renewcommand\r{\mathbb{R}}
\renewcommand\ll\left
\newcommand\rr\right

\begin{document}

\begin{frontmatter}

\title{Well-posedness and propagation of chaos for McKean-Vlasov equations with jumps and locally Lipschitz coefficients}

\runtitle{McKean-Vlasov equations with locally Lipschitz coefficients}

\begin{aug}

\author{\fnms{Xavier} \snm{Erny}\ead[label=e4]{xavier.erny@univ-evry.fr}}

\address{
  Universit\'e Paris-Saclay, CNRS, Univ Evry, Laboratoire de Math\'ematiques et Mod\'elisation d'Evry, 91037, Evry, France\\
	}

\runauthor{X. Erny}

 \end{aug}

\begin{abstract}
We study McKean-Vlasov equations where the coefficients are locally Lipschitz continuous. We prove the strong well-posedness and a propagation of chaos property. These questions are classical under the assumptions that the coefficients are Lipschitz continuous. In the locally Lipschitz case, we use truncation arguments and Osgood's lemma instead of Gr\"onwall's lemma. Technical difficulties appear in the proofs, in particular for the existence of solution of the McKean-Vlasov equations. This proof relies on a Picard iteration scheme that is not guaranteed to converge in an $L^1-$sense. However, we prove its convergence in distribution, and the (strong) well-posedness of the equation
\end{abstract}



\end{frontmatter}

\section{Introduction}

The aim of this paper is to prove the strong well-posedness and a propagation of chaos property for Mckean-Vlasov equations. These are SDEs where the coefficients depend on the solution of the equation and on the law of this solution. This type of equation arises naturally in the framework of $N-$particle systems where the particles interact in a mean field way: this phenomenon can be seen as a law of large numbers. Indeed, in examples where the dynamic of the $N-$particle system is directed by an SDE, the mean field interactions can be expressed as a dependency of the coefficients on the empirical measure of the system. And as $N$ goes to infinity, this empirical measure converges to the law of any particle of the limit system. This entails natural dependencies of the coefficients on the law of the solution of the limit SDE. For instance, see \cite{de_masi_hydrodynamic_2015} and \cite{fournier_toy_2016} for examples in neural network modeling, \cite{fischer_continuous_2016} for an example in portfolio modeling, and \cite{carmona_mean_2016} for an application in mean field games. 

The weak well-posedness and the propagation of chaos are classical for the McKean-Vlasov equations without jump term, even without assuming that the coefficients are Lipschitz continuous. \cite{gartner_mckean-vlasov_1988} treats both questions in this frame. We can also mention more recent work on the well-posedness of McKean-Vlasov equations without jump term as \cite{mishura_existence_2020} and \cite{chaudru_de_raynal_strong_2020}, with different assumptions on the smoothness of the coefficients, and \cite{lacker_strong_2018} that investigates the well-posedness and the propagation of chaos.

In this paper, we consider McKean-Vlasov equations with jumps. The questions about the strong well-posedness and the propagation of chaos have also been studied in this framework under globally Lipschitz assumptions on the coefficients: see \cite{graham_mckean-vlasov_1992} for the well-posedness and \cite{andreis_mckeanvlasov_2018} for the propagation of chaos. Note that in Section~4 of \cite{andreis_mckeanvlasov_2018}, these questions are treated in a multi-dimensional case, where the drift coefficient is of the form $-\nabla b_1(x) + b_2(x,m),$ where $b_1$ is $C^1$ and convex, and $b_2$, as well as the jump coefficient and the volatility coefficient, are globally Lipschitz.



{
The novelty of our results is to work on McKean-Vlasov equations with jumps and with locally Lipschitz coefficients. The local Lipschitz constants of the drift and jump terms are assumed to growth at most linearly, and the Brownian coefficient is assumed to be Lipschitz continuous. In addition, we suppose that the drift and jump coefficients growth at most linearly, and that the Brownian coefficient is bounded (see Assumptions~\ref{hypmoment} and~\ref{hypnonlip} for a precise and complete statement of the hypothesis).
}

The first main result is the strong well-posedness, under this locally Lipschitz assumption, of the following McKean-Vlasov equation
$$dX_t = b(X_t,\mu_t)dt + \sigma(X_t,\mu_t)dW_t + \int_{\r_+\times E} \Phi(X_{t-},\mu_{t-},u)\uno{z\leq f(X_{t-},\mu_{t-})}d\pi(t,z,u),$$
where $\mu_t$ is the distribution of $X_t,$ $W$ a Brownian motion, $\pi$ a Poisson measure and $E$ some measurable space (see the beginning of Section~\ref{sectionwellposedmckeangeneral} for details on the notation). To prove the well-posedness when the coefficients are locally Lipschitz continuous, we adapt the computations of the proofs in the globally Lipschitz continuous case. We use a truncation argument to handle the dependency of the local Lipschitz constant w.r.t. to the variables. On the contrary of the globally Lipschitz case, Gr\"onwall's lemma does not allow to conclude immediately. In the locally Lipschitz case, we have to use a generalization of this lemma: Osgood's lemma (see Lemma~\ref{osgood}). The uniqueness of solution of the McKean-Vlasov equation follows rather quickly from Osgood's lemma and the truncation argument, but other difficulties emerge in the proof of the existence of solution. We construct a weak solution of the equation using a Picard iteration scheme, but the fact that the coefficients are only locally Lipschitz continuous does not allow to prove that this scheme converges in an $L^1-$sense. Instead, we prove that a subsequence converges in distribution to some limit that is shown to be a solution of the equation. Some technical difficulties emerge in this part of the proof for two reasons. The first one is that the Picard scheme is not shown to converge but only to have a converging subsequence. This implies that we need to control the variation between two consecutive steps of the scheme. The second one is that we only prove a convergence in distribution, thus it is not straightforward that the limit of the Picard scheme is solution to the equation. It is shown studying its semimartingale characteristics.

The second main result is a propagation of chaos property of McKean-Vlasov particle systems under the same locally Lipschitz assumptions. More precisely, it is the convergence of the following $N-$particle system
\begin{align*}
dX^{N,i}_t=& b(X^{N,i}_t,\mu^N_t)dt + \sigma(X^{N,i}_t,\mu^N_t)dW^i_t + \int_{\r_+\times F^{\n^*}}\Psi(X^{N,i}_{t-},\mu^N_{t-},v^i)\uno{z\leq f(X^{N,i}_{t-},\mu^N_{t-})}d\pi^i(t,z,v)\\
&+ \frac1N\sum_{\substack{j=1\\j\neq i}}^N\int_{\r_+\times F^{\n^*}}\Theta(X^{N,j}_{t-},X^{N,i}_{t-},\mu^N_{t-},v^j,v^i)\uno{z\leq f(X^{N,j}_{t-},\mu^N_{t-})}d\pi^j(t,z,v),
\end{align*}
where $\mu^N_t:=N^{-1}\sum_{j=1}^N\delta_{X^{N,j}_t}$, to the infinite system
\begin{align*}
d\bar X^{i}_t=& b(\bar X^{i}_t,\bar \mu_t)dt + \sigma(\bar X^{i}_t,\bar \mu_t)dW^i_t + \int_{\r_+\times F^{\n^*}}\Psi(\bar X^{i}_{t-},\bar\mu_{t-},v^i)\uno{z\leq f(\bar X^{i}_{t-},\bar\mu_{t-})}d\pi^i(t,z,v)\\
&+ \int_\r\int_{F^{\n^*}}\Theta(x,\bar X^{i}_{t},\bar \mu_{t},v^1,v^2)f(x,\bar\mu_t)d\nu(v)d\bar\mu_t(x),
\end{align*}
where $\bar \mu_t:=\mathcal{L}(\bar X_t)$, as $N$ goes to infinity. The $W^i$ ($i\geq 1$) are independent standard Brownian motions, the $\pi^i$ ($i\geq 1$) are independent Poisson measures, $F$ is some measurable space and $\n^*$ denotes the set of the positive integers (see Section~\ref{sectionchaosmckeangeneral} for details on the notation). The proof of this second main result relies on a similar reasoning as the one used to prove the uniqueness of the McKean-Vlasov equation: a truncation argument and Osgood's lemma.

Let us note that this propagation of chaos property has already been proven under different hypothesis. Indeed, the $N-$particle system and the limit system above are the same as in \cite{andreis_mckeanvlasov_2018}. Note also that, in this model, for each $N\in\n^*$, the particles $X^{N,i}_t$ ($1\leq i\leq N$) do not only interact through the empirical measure $\mu^N_t,$ but also through the simultaneous jumps term.

Let us remark that in Assumption~\ref{hypmoment}, we also assume the initial condition to admit finite exponential moments. This property is used to obtain a priori estimates on the exponential moments of the solutions of the McKean-Vlasov equation (see Lemma~\ref{apriorimckeangeneral}). This is important in the proof of Theorem~\ref{wellposedmckeangeneral} to do the truncation arguments mention above.


Let us finally mention that, in this paper, we chose to work in dimension one to simplify the notation, but the results still hold in finite dimensions.

{\bf Organization.} In Section~\ref{sectionwellposedmckeangeneral}, we state and prove our first main result: the well-posedness of the McKean-Vlasov equation~\eqref{mckeangeneral} with locally Lipschitz coefficients. Section~\ref{sectionchaosmckeangeneral} is devoted to our second main result, the propagation of chaos in the same framework.

\subsection{Notation}

Let us introduce some notation we use throughout the paper:
\begin{itemize}
\item If $X$ is random variable, we denote by $\mathcal{L}(X)$ its distribution.
\item If $X$ and $X_n$ ($n\in\n^*$) are random variables, we denote by $X_n\stackrel{\mathcal{L}}{\underset{n\rightarrow+\infty}{\longrightarrow}} X$ for "$(X_n)_n$ converges in distribution to $X$".
\item $\mathcal{P}_1(\r)$ is the space of probability measures on~$\r$ with finite first moment. This space will always be endowed with the first-order Wassertein metric~$W_1$ defined by: for $m_1,m_2\in\mathcal{P}_1(\r),$
$$W_1(m_1,m_2)=\underset{X_1\sim m_1,X_2\sim m_2}{\inf}\esp{|X_1-X_2|}=\underset{f\in Lip_1}{\sup}\int_\r f(x)dm_1(x) - \int_\r f(x)dm_2(x),$$
with $Lip_1$ the space of Lipschitz continuous functions with Lipschitz constant non-greater than one. Let us note that characterizations of this convergence are given in Theorem~6.9 and Definition~6.8 of \cite{villani_optimal_2008}.
\item For $T>0$ and $(G,d)$ a Polish space, $D([0,T],G)$ (resp. $D(\r_+,G)$) denotes the space of c\`adl\`ag $G-$valued functions defined on $[0,T]$ (resp. $\r_+$) endowed with Skorohod topology, whence this space is Polish. Let us recall that the convergence of a sequence $(x_n)_n$ of $D([0,T],G)$ to some $x$ in Skorohod topology is equivalent to the existence of continuous increasing functions $\lambda_n$ satisfying $\lambda_n(0)=0,$ $\lambda_n(T)=T$ and both
$$\underset{0\leq t\leq T}{\sup}|\lambda_n(t)-t|\textrm{ and }\underset{0\leq t\leq T}{\sup}d(x(\lambda_n(t)),x_n(t))$$
vanish as $n$ goes to infinity. In the following, we call such a sequence $(\lambda_n)_n$ a sequence of time-changes.
\item $L$ denotes the Lipschitz constant of the coefficients (see Assumption~\ref{hypnonlip}), $a$ a positive constant defined in Assumption~\ref{hypmoment}, and $k_0\in\n$ in Assumption~\ref{hypnonlip}.
\item $C$ denotes any arbitrary positive constant, whose value can change from line to line in an equation. If the constant depends on some parameter $\theta,$ we write $C_\theta$ instead.
\end{itemize}

\section{Well-posedness of McKean-Vlasov equations}
\label{sectionwellposedmckeangeneral}

This section is dedicated to prove the well-posedness of the following McKean-Vlasov equation.

\begin{equation}
\label{mckeangeneral}
dX_t = b(X_t,\mu_t)dt + \sigma(X_t,\mu_t)dW_t + \int_{\r_+\times E} \Phi(X_{t-},\mu_{t-},u)\uno{z\leq f(X_{t-},\mu_{t-})}d\pi(t,z,u),
\end{equation}
with $\mu_t=\mathcal{L}(X_t),$ $W$ a standard one-dimensional Brownian motion, $\pi$ a Poisson measure on $\r_+\times\r_+\times E$ having intensity $dt\cdot dz\cdot d\rho(u),$ where $(E,\mathcal{E},\rho)$ is a $\sigma-$finite measure space. The assumptions on the coefficients are specified in Assumptions~\ref{hypmoment} and~\ref{hypnonlip} below. Let us note here that $f$ is assumed to be non-negative. {The conditions specified in Assumption~\ref{hypmoment} allows to prove that any solution of~\eqref{mckeangeneral} admits a priori some exponential moments (cf Lemma~\ref{apriorimckeangeneral})}.

{
\begin{assu}\label{hypmoment}$ $
\begin{enumerate}
\item Growth condition: there exists some $C>0$ such that for any $x\in\r,m\in\mathcal{P}_1(\r),n\in\n^*,$
\begin{align*}
&|b(x,m)|\leq C\ll(1+|x|\rr),\\
&|\sigma(x,m)|\leq C,\\
&\int_E|\ll(x+\Phi(x,m,u))^n-x^n\rr|f(x,m) d\rho(u) \leq C\cdot n \ll(1+|x|^n\rr).
\end{align*}
\item Initial condition: there exists some constant~$a>0$ such that
$$\esp{e^{a|X_0|}}<\infty.$$
\end{enumerate}
\end{assu}
}

{
\begin{ex}\label{resetjump}
The above conditions on the functions~$b$ and~$\sigma$ are quite usual, but not the one on the jump term. Let us give an example in which it is satisfied. Let $\Phi$ be defined as: for all $x\in\r,m\in\mathcal{P}_1(\r),u\in\r,$
$$\Phi(x,m,u) := -x + \phi(m,u),$$
where $\phi:\mathcal{P}_1(\r)\times E\rightarrow \r.$ If the function $f$ is bounded and the function $\phi$ is $[-1,1]-$valued, then the condition of Assumption~\ref{hypmoment} on the jump term is satisfied.

This kind of function~$\Phi$ appears naturally in the models where the jump term is a kind of reset term. For example, in neurosciences, if $(X_t)_t$ models the membrane potential of a neurone, and if the jump times are the times at which the neuron emits a spike, then at its jump times the potential $X$ has to be reset at its resting value~0. This kind of model has been studied in \cite{fournier_toy_2016} and in~\cite{erny_conditional_2021} (with $\phi\equiv 0$). In addition the term $\phi(m,u)$ can be seen as a noise term, where $u$ is a random variable with law~$\rho.$
\end{ex}
}

{
\begin{rem}
Technically, Assumption~\ref{hypmoment} is only used to obtain some control on the processes that we manipulate: Lemma~\ref{apriorimckeangeneral},~\eqref{aprioripicardmckeangeneral} and Lemma~\ref{apriorixnmckeangeneral}. So the result of Theorem~\ref{wellposedmckeangeneral} (resp. Theorem~\ref{chaosmckeangeneral}) still hold true if Assumption~\ref{hypmoment} is replaced by any another under which Lemma~\ref{hypmoment} and~\eqref{aprioripicardmckeangeneral} hold true (resp. Lemma~\ref{apriorixnmckeangeneral} holds true).
\end{rem}
}

\begin{assu}$ $
\label{hypnonlip}
\begin{enumerate}
\item Locally Lipschitz conditions: there exists $k_0\in\n,$ such that for all $x_1,x_2\in\r,m_1,m_2\in\mathcal{P}_1(\r),$
\begin{multline*}
|b(x_1,m_1)-b(x_2,m_2)| + \int_E\int_{\r_+}|\Phi(x_1,m_1,u)\uno{z\leq f(x_1,m_1)} - \Phi(x_2,m_2,u)\uno{z\leq f(x_2,m_2)}|dzd\rho(u)\\
\leq L\ll(1+|x_1|+|x_2| + { \int_\r |x|^{k_0}dm_1(x) + \int_\r |x|^{k_0}dm_2(x)}\rr)\ll(|x_1-x_2|+W_1(m_1,m_2)\rr).
\end{multline*}
\item Globally Lipschitz condition for~$\sigma:$ for all $x_1,x_2\in\r,m_1,m_2\in\mathcal{P}_1(\r),$
$$|\sigma(x_1,m_1)-\sigma(x_2,m_2)|\leq L\ll(|x_1-x_2| + W_1(m_1,m_2)\rr).$$
\end{enumerate}
\end{assu}

{ Let us remark that in Assumption~\ref{hypnonlip}, if for example $k_0=2,$ then the locally Lipschitz conditions are trivially satisifed for measures $m_1,m_2\in\mathcal{P}_1(\r)$ with infinite second moment. Indeed, the RHS of the inequality is infinite while the LHS is finite.}

\begin{rem}
\label{remhypmckeangeneral}
If we consider equation~\eqref{mckeangeneral} without the jump term (that is $\Phi\equiv 0$), then, we can adapt the proof of Theorem~\ref{wellposedmckeangeneral} to the case where $\sigma$ is also locally Lipschitz continuous. More precisely, we can replace the two first Items of Assumption~\ref{hypnonlip} by: for all $x_1,x_2\in\r,m_1,m_2\in\mathcal{P}_1(\r),$
\begin{multline*}
|b(x_1,m_1)-b(x_2,m_2)| + |\sigma(x_1,m_1)-\sigma(x_2,m_2)|\\
\leq L\ll(1+\sqrt{|x_1|}+\sqrt{|x_2|} + \sqrt{\int_\r |x|^{k_0}dm_1(x)} + \sqrt{\int_\r |x|^{k_0}dm_2(x)}\rr)\ll(|x_1-x_2|+W_1(m_1,m_2)\rr).
\end{multline*}
See Remark~\ref{rem2hypmckeangeneral} for more details on the adaptation of the proof.
\end{rem}


{
\begin{rem}
A sufficient condition to obtain the locally Lipschitz condition of the jump term in Assumption~\ref{hypnonlip} for $k_0=1$ is given by the following direct conditions on the functions $f$ and $\Phi:$ there exists some positive constant~$C$ such that for all $x_1,x_2\in\r,m_1,m_2\in\mathcal{P}_1(\r),$
\begin{align*}
|f(x_1,m_1) - f(x_2,m_2)|\leq& C \ll(|x_1-x_2| + W_1(m_1,m_2)\rr),\\
\int_E \ll|\Phi(x_1,m_1,u)-\Phi(x_2,m_2,u)\rr|d\rho(u)\leq & C \ll(|x_1-x_2| + W_1(m_1,m_2)\rr).
\end{align*}
\end{rem}
}

\begin{ex}
A natural form of the coefficient of a McKean-Vlasov equation is given by the so-called "true McKean-Vlasov" case. In our framework, it is possible to consider the function~$b$ under the following form
$$b(x,m)= \int_\r \tilde b(x,y)dm(y),$$
with $\tilde b:\r^2\rightarrow\r.$

For $b$ to satisfy the Lipschitz condition of Assumption~\ref{hypnonlip} in this example, it is sufficient to assume that: for all $x,x',y,y'\in\r,$ 
$$|\tilde b(x,y) - \tilde b(x',y')|\leq C(1+|x|+|x'|)(|x-x'| + |y-y'|).$$
Indeed, for any $x,x'\in\r,m,m'\in\mathcal{P}_1(\r),$
\begin{align*}
|b(x,m)-b(x',m')|\leq & |b(x,m) - b(x',m)| + |b(x',m) - b(x',m')|\\
\leq & \int_\r |\tilde b(x,y) - \tilde b(x',y)|dm(y) + \ll|\int_\r \tilde b(x',y)dm(y) - \int_\r \tilde b(x',y)dm'(y)\rr|\\
\leq & C(1+|x|+|x'|)|x-x'| + C(1+2|x'|)W_1(m,m'),
\end{align*}
where the second quantity of the last line has been obtained using Kantorovich-Rubinstein duality (see Remark~6.5 of \cite{villani_optimal_2008}) and the fact that, for a fixed $x'$, the function $y\mapsto\tilde b(x',y)$ is Lipschitz continuous with Lipschitz constant $C(1+2|x'|).$
\end{ex}

{
\begin{ex}
Let us mention that the conditions of Assumption~\ref{hypnonlip} appears naturally when one studies some mean field limits of particle systems. For example, in \cite{fournier_toy_2016} and \cite{erny_conditional_2021}, the authors study models where the form of the jump term is the one described in Example~\ref{resetjump}, and in particular, the jump term leads to a framework where the coefficients are not globally Lipschitz. In these two papers, the results are proved under an Assumption that guarantees the function $x\mapsto xf(x)$ (without dependency w.r.t. the measure variable~$m$) to be globally Lipschitz w.r.t. some appropriate metrics.

Note that the locally Lipschitz condition for the drift term is also natural in the mean field limit frame. Indeed, in this framework, a jump term in the SDEs of some $N-$particles system can become a drift term in the limit system (as $N$ goes to infinity) where the corresponding drift coefficient is the product of the jump height function with the jump rate function: see for instance the last term of the equations~\eqref{xnmckeangeneral} and~\eqref{barxmckeangeneral} in Section~\ref{sectionchaosmckeangeneral}. 
\end{ex}
}

\begin{thm}
\label{wellposedmckeangeneral}
Under Assumptions~\ref{hypmoment} and~\ref{hypnonlip}, there exists a unique strong solution of~\eqref{mckeangeneral}.
\end{thm}

The rest of this section is dedicated to prove Theorem~\ref{wellposedmckeangeneral}.

\subsection{A priori estimates for equation~\eqref{mckeangeneral}}

In this section, we prove the following a priori estimates for the solutions of the SDE~\eqref{mckeangeneral}.

\begin{lem}\label{apriorimckeangeneral}
Grant Assumption~\ref{hypmoment}. Any solution $(X_t)_{t\geq 0}$ of~\eqref{mckeangeneral} satisfies for all $t> 0:$
\begin{itemize}
\item[(i)] there exists some $0<\alpha_t<a$ (with $a$ the same constant as in Assumption~\ref{hypmoment}) such that
$$\underset{0\leq s\leq t}{\sup}\esp{e^{\alpha_t|X_s|}}<\infty,$$ 
\item[(ii)] $\esp{\underset{0\leq s\leq t}{\sup}|X_s|}<\infty.$
\end{itemize}
\end{lem}

\begin{proof}
{
Let us prove~$(i).$ In a first time, we show that for any $n\in\n, t\geq 1,$
\begin{equation}\label{controlXtn}
\esp{|X_t|^n} \leq \ll(\esp{|X_0|^n} + C t n^{n/2 +1}\rr) e^{C n t},
\end{equation}
where $C$ is positive constant independent of $n.$

For $n=0$, there is nothing to prove, and for $n=1,$ the following estimate is classical with standard techniques: for all $t>0,$
$$\esp{|X_t|}\leq \ll(\esp{|X_0|} + C(t+\sqrt{t})\rr) e^{C(t+\sqrt{t})},$$
where $\sqrt{t}$ comes from Burkholder-Davis-Gundy's inequality. Then, assuming that $t\geq 1,$ we have $\sqrt{t}\leq t,$ whence~\eqref{controlXtn} holds true for~$n=1.$

Let $n\geq 2.$ Let
$$\psi_n(x) := |x|^n.$$

Note that $\psi_n$ is $C^2,$ and that, for any $x\in\r,$
$$|\psi'_n(x)| = n|x|^{n-1}\textrm{ and }|\psi''_n(x)| = n(n-1)|x|^{n-2}.$$

By Ito's formula, for any $t>0,$
\begin{align*}
\psi_n(X_t)=& \psi_n(X_0) +\int_0^t \psi_n'(X_s)b(X_s,\mu_s)ds + \int_0^t \psi'_n(X_s)\sigma(X_s,\mu_s)dW_s + \frac{1}{2}\int_0^t \psi''_n(X_s)\sigma(X_s,\mu_s)^2ds\\
& + \int_{[0,t]\times\r_+} \ll[\psi_n\ll(X_{s-} + \Phi(X_{s-},\mu_{s-},u)\rr)- \psi_n(X_{s-})\rr] \uno{z\leq f(X_{s-},\mu_{s-})}d\pi(s,z,u).
\end{align*}
Then, for any fixed $M>0,$ introducing $\tau_M := \inf\{t>0 : |X_t|>M\},$ we have
\begin{align}
\esp{|X_{t\wedge\tau_M}|^n}\leq& \esp{|X_0|^n} + n\int_0^t \esp{|X^{n-1}_{s\wedge\tau_M} b(X_{s\wedge\tau_M},\mu_{s\wedge\tau_M})|}ds\nonumber\\
&+ \frac{n(n-1)}{2}\int_0^t \esp{|X^{n-2}_{s\wedge\tau_M} \sigma(X_{s\wedge\tau_M},\mu_{s\wedge\tau_M})^2|}ds\nonumber\\
&+\int_0^t \esp{\int_\r \ll|\ll[\ll|X_{s\wedge\tau_M} + \Phi(X_{s\wedge\tau_M},\mu_{s\wedge\tau_M},u)\rr|^n - |X_{s\wedge\tau_M}|^n\rr]\rr|f(X_{s\wedge\tau_M},\mu_{s\wedge\tau_M})d\rho(u)}ds\nonumber\\
\leq& \esp{|X_0|^n} + Cnt + Cn\int_0^t \esp{|X_{s\wedge\tau_M}|^n}ds + Cn(n-1)\int_0^t\esp{|X_{s\wedge\tau_M}|^{n-2}}ds,\label{controlaux}
\end{align}
where we have used Assumption~\ref{hypmoment} and the fact that, for any $x\in\r,$ $|x|^{n-1} + |x|^n\leq 1 + 2|x|^n$ to control the drift term.

One problem in~\eqref{controlaux} is the $n(n-1)$ in the last term. Applying Gr\"onwall's lemma at this step (by bounding $|X_{s\wedge\tau_M}|^{n-2}$ by $1+|X_{s\wedge\tau_M}|^{n}$) would lead to obtain a term $n(n-1)$ in the exponential of~\eqref{controlXtn}. To avoid this problem, we use that for any $x\geq 0,n\in\n^*,$
$$x^{n-2} = x^{n-2}\cdot\uno{x^2\leq n} + x^{n-2}\cdot\uno{x^2>n}\leq n^{(n-2)/2} + x^n\cdot x^{-2}\cdot\uno{x^2>n} \leq n^{(n-2)/2} + \frac{x^n}{n},$$

we deduce from the previous computation that
\begin{equation}\label{xpuisn}
\esp{|X_{t\wedge\tau_M}|^n}\leq \esp{|X_0|^n} + C t n^{n/2 + 1} + Cn\int_0^t \esp{|X_{s\wedge\tau_M}|^n}ds.
\end{equation}

By Gr\"onwall's lemma, for any $t>0,$
$$\esp{|X_{t\wedge\tau_M}|^n} \leq \ll(\esp{|X_0|^n} + C t n^{n/2 +1}\rr) e^{C n t}.$$

As the bound above does not depend on~$M$, this implies that $\tau_M$ goes to infinity almost surely as $M$ goes to infinity. Whence Fatou's lemma allows to prove~\eqref{controlXtn}.

Now, let us show the point~$(i)$. Let $t\geq 1$ be fixed, and let us consider some $\alpha_t>0$ whose value will be chosen later. By~\eqref{controlXtn}, for any $s\leq t,$
\begin{align}
\esp{e^{\alpha_t |X_s|}}=& \sum_{n=0}^\infty \frac{\alpha_t^n}{n!}\esp{|X_s|^n}\leq \sum_{n=0}^\infty\frac{1}{n!}\alpha_t^n e^{Cns}\esp{|X_0|^n} + Cs\sum_{n=0}^\infty \frac{\alpha_t^n}{n!} n^{n/2 +1} e^{Cns}\nonumber\\
\leq& \esp{\sum_{n=0}^\infty\frac{1}{n!}\ll(\alpha_t e^{Cs} |X_0|\rr)^n} + Cs\sum_{n=0}^\infty\frac{\alpha_t^n}{n!}n^{n/2 + 1} e^{Cns}\nonumber\\
\leq & \esp{\exp\ll(\alpha_t e^{Cs} |X_0|\rr)} + C t\sum_{n=0}^\infty\frac{\alpha_t^n}{n!}n^{n/2 + 1} e^{Cnt}\label{eaxt}
\end{align}

The first sum in the last line above is finite provided that $\alpha_t < ae^{-Ct}.$ To see that the second sum is finite, it is sufficient to use D'Alembert's criterion: let
$$u_n = \frac{\alpha_t ^n}{n!} n^{n/2 + 1} e^{C n t},$$
and remark that
$$\frac{u_{n+1}}{u_n} = \alpha_t \frac{1}{n+1} e^{C t} \frac{(n+1)^{n/2 +1 +1/2}}{n^{n/2+1}} = \alpha_t e^{Ct} (n+1)^{-1/2} \ll(1+\frac1n\rr)\sqrt{\ll(1+\frac1n\rr)^n}\underset{n\rightarrow\infty}{\longrightarrow}0.$$

Then, the point~$(i)$ of the lemma is a consequence of~\eqref{eaxt}.
	
To prove the point~$(ii)$, let us use the following bound, which is a direct consequence of the form of the SDE~\eqref{mckeangeneral},
\begin{multline*}
\esp{\underset{0\leq s\leq t}{\sup}|X_s|}\leq \esp{|X_0|} + \int_0^t \esp{|b(X_s,\mu_s)|}ds + \esp{\underset{0\leq s\leq t}{\sup}\ll|\int_0^s \sigma(X_t,\mu_r)dW_r\rr|}\\
+ \int_0^t\esp{\int_E|\Phi(X_s,\mu_s,u)|f(X_s,\mu_s)d\rho(u)}ds.
\end{multline*}

Then using Burkholder-Davis-Gundy's inequality and Assumption~\ref{hypmoment}, we have
$$\esp{\underset{0\leq s\leq t}{\sup}|X_s|}\leq \esp{|X_0|} + C(t+\sqrt{t}) + C\int_0^t \esp{|X_s|}ds.$$

Finally, the point~$(ii)$ follows from the point~$(i)$ of the lemma.
}
\end{proof}

\subsection{Pathwise uniqueness for equation~\eqref{mckeangeneral}}
\label{uniquenessmckeangeneral}


\begin{prop}\label{unicitemckeangeneral}
Grant Assumptions~\ref{hypmoment} and~\ref{hypnonlip}. The pathwise uniqueness property holds true for~\eqref{mckeangeneral}.
\end{prop}

\begin{proof}
Let $(\hat X_t)_{t\geq 0}$ and $(\check X_t)_{t\geq 0}$ be two solutions of~\eqref{mckeangeneral} defined w.r.t. the same initial condition~$X_0,$ the same Brownian motion~$W$ and the same Poisson measure~$\pi$. The proof consists in showing that the function
$$u(t) := \esp{\underset{0\leq s\leq t}{\sup}|\hat X_s-\check X_s|}$$
is zero. This choice of function~$u$ is inspired of the proof of Theorem~2.1 of \cite{graham_mckean-vlasov_1992}. {It allows to treat equations with both a jump term and a Brownian term since the classical $L^1$ norm would cause problems with the Brownian term, and the classical $L^2$ norm with the jump term.}

We know that, for all $t\geq 0,$ $u_t<\infty$ by Lemma~\ref{apriorimckeangeneral}$.(ii).$

Writing $\hat\mu_t := \mathcal{L}(\hat X_t)$ and $\check\mu_t:=\mathcal{L}(\check X_t),$ we have
\begin{multline*}
\ll|\hat X_s - \check X_s\rr|\leq \int_0^s |b(\hat X_r,\hat\mu_r) - b(\check X_r,\check\mu_r)|dr + \ll|\int_0^s (\sigma(\hat X_r,\hat\mu_r) - \sigma(\check X_r,\check\mu_r))dW_r\rr|\\
+ \int_{[0,s]\times\r_+\times\r} \ll|\Phi(\hat X_{r-},\hat\mu_{r-},u)\uno{z\leq f(\hat X_{r-},\hat\mu_{r-})}-\Phi(\check X_{r-},\check\mu_{r-},u)\uno{z\leq f(\check X_{r-},\check\mu_{r-})}\rr|d\pi(r,z,u).
\end{multline*}

This implies that
\begin{multline}\label{osgoodutpre}
u(t)\leq  L\esp{\ll(\int_0^t \ll(|\hat X_s - \check X_s| + W_1(\hat\mu_s,\check\mu_s)\rr)^2ds\rr)^{1/2}} \\
+2L\int_0^t\esp{\ll(1+|\hat X_s|+|\check X_s|+\int_\r |x|^{k_0}d\hat\mu_s(x) + \int_\r |x|^{k_0}d\check\mu_s(x)\rr)\ll(|\hat X_s - \check X_s| + W_1(\hat\mu_s,\check\mu_s)\rr)}ds
\end{multline}
where we have used Burkholder-Davis-Gundy's inequality to deal with the Brownian term that corresponds to the term at the first line above. The term at the second line corresponds to the controls of the drift term and the jump term.

Now let us fix some $T>0$ (later $T$ will be fixed at $T=1/(16L^2)$). By Lemma~\ref{apriorimckeangeneral}$.(i)$, we have,
\begin{equation*}
\underset{0\leq s\leq T}{\sup}\int_\r e^{\alpha_T|x|}d\hat\mu_s(x) + \underset{0\leq s\leq T}{\sup}\int_\r e^{\alpha_T|x|}d\check\mu_s(x)\leq C_T<\infty,
\end{equation*}
whence
$$\underset{0\leq s\leq T}{\sup}\int_\r |x|^{k_0}d\hat\mu_s(x) + \underset{0\leq s\leq T}{\sup}\int_\r |x|^{k_0}d\check\mu_s(x)\leq C_T<\infty.$$

And, from the definition of $W_1,$ we have the following bound
$$W_1(\hat\mu_s,\check\mu_s)\leq \esp{\ll|\hat X_s-\check X_s\rr|}\leq u(s).$$

Then, \eqref{osgoodutpre} and Lemma~\ref{apriorimckeangeneral} imply that, for all $0\leq t\leq T,$
\begin{align*}
u(t)\leq& \int_0^t\esp{\ll(1+|\hat X_s|+|\check X_s|+C_T\rr)\ll(|\hat X_s - \check X_s| + u(s)\rr)}ds \\
&+ L\esp{\ll(\int_0^t \ll(|\hat X_s - \check X_s| + u(s)\rr)^2ds\rr)^{1/2}}\\
\leq& C_T\int_0^t\esp{\ll(1+|\hat X_s| + |\check X_s|\rr)\ll(|\hat X_s-\check X_s|+u(s)\rr)}ds + 2L\sqrt{t}u(t)\\
\leq& C_T\int_0^t\esp{\ll(1+|\hat X_s| + |\check X_s|\rr)\ll(|\hat X_s-\check X_s|\rr)}ds+C_T\int_0^t u(s)ds + 2L\sqrt{t}u(t),
\end{align*}
where we have bounded the second integral of the RHS of the first inequality above by $t$ times the supremum of the integrand. Note that the value of $C_T$ changes from line to line.

Now, to end the proof, we have to control a term of the type $(1+|x|+|y|)|x-y|.$ To do so, we use a truncation argument based on the following inequality: for all $x,y\in\r,R>0,$
$$(1+|x|+|y|)|x-y|\leq (1+2R)|x-y| + (1+|x|+|y|)|x-y|\ll(\uno{|x|>R}+\uno{|y|>R}\rr).$$

Let $R:s\mapsto R_s>0$ be the truncation function whose values will be chosen later. By Lemma~\ref{apriorimckeangeneral}, for any $0\leq s\leq T,$
$$\esp{\ll(1+|\hat X_s|+|\check X_s|\rr)\ll|\hat X_s-\check X_s\rr|}\leq (1+2R_s)u(s) + C_T\sqrt{\pro{|\hat X_s|>R_s}} + C_T\sqrt{\pro{|\check X_s|>R_s}}.$$

The exponential moments proven in Lemma~\ref{apriorimckeangeneral}$.(i)$ are used to control the two last term above. Indeed, by Markov's inequality
$$\pro{|\hat X_s|>R_s} + \pro{|\check X_s|>R_s}\leq C_Te^{-\alpha_T R_s}.$$

Consequently, defining $r_s := \alpha_T R_s/2$, for any $0\leq t\leq T,$
$$u(t)\leq C_T\int_0^t \ll[(1+r_s)u(s) + e^{-r_s}\rr]ds + 2L\sqrt{t}u(t).$$

Now, let $T=1/(16L^2)$ such that $2L\sqrt{T}\leq 1/2.$ Then, we can rewrite the above inequality as, for all $t\in[0,T],$
$$u(t)\leq C_T\int_0^t\ll[(1+r_s)u(s) + e^{-r_s}\rr]ds.$$

Let us prove by contradiction that, for all $t\leq T,$ $u(t)=0.$ To do so, let $t_0:=\inf\{t>0:u(t)>0\}$ and assume that $t_0<T.$ Notice that, as $u$ is non-decreasing, this implies that, for all $t\in[0,t_0[,$ $u(t)=0.$ In particular, for all $t\in[t_0,T],$
$$u(t)\leq C_T\int_{t_0}^t\ll[(1+r_s)u(s) + e^{-r_s}\rr]ds.$$

Besides $u(t)$ is finite and bounded (see Lemma~\ref{apriorimckeangeneral}$.(ii)$) on $ [0, T ], $ say by a constant $ D>1.$ Let $v(t):=u(t)/(De^{2})<e^{-2}.$ Obviously $v$ satisfies the same inequality as $u$ above. Now we define $r_s := -\ln v(s),$ so that, for all $t_0 < t \le T, $
$$v(t)\leq C_T\int_{t_0}^t(2-\ln v(s))v(s)ds\leq -2C_T\int_{t_0}^t v(s)\ln v(s)ds,$$
where we have used that $ 2 -  \ln v(s) \le - 2 \ln v(s) $ since $-\ln v(s)\geq 2.$

In particular, for any $c\in]0,e^{-2}[,$ for all $ t_0 \le t \le T,$ 
$$ v(t)\leq  c  -2C_T\int_{t_0}^t v(s)\ln v(s)ds.$$ 

Then, introducing $M: x\in]0,e^{-2}[\mapsto \int_x^{e^{-2}}\frac{1}{-s\ln s}ds,$ we may apply Osgood's lemma (see Lemma~\ref{osgood}) with $ \gamma \equiv 2 C_T $ and $ \mu (v) = (- \ln v) v $ to obtain that 
$$-M(v(T)) + M(c)\leq \int_{t_0}^T 2C_Tds  = 2C_T(T-t_0)$$
or equivalently, 
$$ M(c) \le M( v(T) ) + 2 C_T T .$$ 
Recalling that we assumed $v(T)>0$ such that the right hand side of the above equality is finite,  if we let $c$ tend to 0, we obtain
$$M(0) = \int_0^{e^{-2}}\frac{1}{-s\ln s}ds\leq \int_{v(T)}^{e^{-2}}\frac{1}{-s\ln s}ds + 2C_TT <\infty,$$
which is absurd since $M(0)  = \infty .$

A consequence of the above considerations is that for all $t\in[0,T]$, $u(t)=0.$ Recalling the definition of $u$, we have proven that the processes $\hat X$ and $\check X$ are equal on~$[0,T]$.

We can repeat this argument on the interval $[T,2T]$  and iterate up to any finite time interval $[0,T_0]$ since $T=1/(16L^2)$ does only depend on the coefficients of the system but not on the initial condition. This proves the pathwise-uniqueness property for the McKean-Vlasov equation~\eqref{mckeangeneral}.
\end{proof}

Let us complete our previous Remark~\ref{remhypmckeangeneral}.

\begin{rem}\label{rem2hypmckeangeneral}
The adaptation suggested in Remark~\ref{remhypmckeangeneral} is the following: in the proof of Proposition~\ref{unicitemckeangeneral} above, one has to replace the distance $\esp{\sup_{s\leq t}|\hat X_s-\check X_s|}$ by $\esp{(\hat X_t-\check X_t)^2},$ and one has to do similar changes in the proof of Proposition~\ref{propexistenceweaksolutionmckeangeneral} below.
\end{rem}

\subsection{Existence of a weak solution of equation~\eqref{mckeangeneral}}
\label{existencemckeangeneral}

Before proving the existence of solution of~\eqref{mckeangeneral}, let us state an elementary lemma about series whose proof is postponed to the Appendix.

\begin{lem}
\label{seriecv0}
Let $(u_n)_{n\geq 0}$ be a sequence of non-negative real numbers, and $S_n = \sum_{k=0}^nu_k$ ($n\in\n$). If the sequence $S_n/n$ vanishes, then there exists a subsequence of $(u_n)_{n\geq 0}$ that converges to $0.$
\end{lem}




The aim of this section is to construct a weak solution of the McKean-Vlasov equation~\eqref{mckeangeneral}, using a Picard iteration. The idea of the proof is to show that this scheme converges to a solution of~\eqref{mckeangeneral}. However, because of our locally Lipschitz conditions, we cannot prove it directly. Instead, we prove that a subsequence converges in distribution by tightness. That is why, in a first time, we only construct a weak solution.

\begin{prop}
\label{propexistenceweaksolutionmckeangeneral}
Grant Assumptions~\ref{hypmoment} and~\ref{hypnonlip}. There exists a weak solution of~\eqref{mckeangeneral} on $[0,T],$ with $T=1/(16L^2).$
\end{prop}

\begin{proof}
As in the proof of the pathwise uniqueness of Section~\ref{uniquenessmckeangeneral}, we work on a time interval~$[0,T]$ where $T>0$ is a number whose value can be fixed at $1/(16L^2)$.

{\it Step~1.} In this first step, we introduce the iteration scheme, and state its basic properties at~\eqref{aprioripicardmckeangeneral}. Let $X^{[0]}_t := X_0,$ and define the process $X^{[n+1]}$ from $\mu^{[n]}_t := \mathcal{L}(X^{[n]}_t)$ by

{
\begin{align}
X^{[n+1]}_t :=& X_0 + \int_0^t b(X^{[n+1]}_s,\mu^{[n]}_s)ds + \int_0^t\sigma(X^{[n+1]}_s,\mu^{[n]}_s)dW_s \label{Xnplus1}\\
&+ \int_{[0,t]\times\r_+\times\r}\Phi(X^{[n+1]}_{s-},\mu^{[n]}_{s-},u)\uno{z\leq f(X^{[n+1]}_{s-},\mu^{[n]}_{s-})}d\pi(s,z,u)\nonumber
\end{align}

We know that $X^{[n+1]}$ is well-defined since the above equation is a classical Ito-SDE (see for example Theorem~$III.2.32$ of \cite[p. 158]{jacod_limit_2003}). Indeed, since we assume that $X^{[n]}$ is well-defined, the function $t\mapsto \mu^{[n]}_t$ is a well-defined deterministic function that does not depend on~$X^{[n+1]}$. More precisely, we can write $b(x,\mu^{[n]}_s) =: \tilde b_n(s,x)$ and write similarly the other coefficients.
}

Then, using the same computations as in the proof of Lemma~\ref{apriorimckeangeneral}, we can prove that, for all $t\geq 0,$
\begin{equation}
\label{aprioripicardmckeangeneral}
\underset{n\in\n}{\sup}~\underset{0\leq s\leq t}{\sup}\esp{e^{\alpha_T |X^{[n]}_s|}}<\infty\textrm{ and }\underset{n\in\n}{\sup}~\esp{\underset{0\leq s\leq t}{\sup}|X^{[n]}_s|}<\infty.
\end{equation}

{ Note that, to obtain~\eqref{aprioripicardmckeangeneral} with the same reasoning as in the proof of Lemma~\ref{apriorimckeangeneral}, it is important in~\eqref{Xnplus1} to use $X^{[n+1]}$ in the coefficients of the SDE and not $X^{[n]}.$}

{\it Step~2.} Now let us show that $(X^{[n]},X^{[n+1]})_n$ has a converging subsequence in distribution in~$D([0,T],\r^2),$ by showing that it satisfies Aldous' tightness criterion:  
\begin{itemize}
\item[$(a)$] for all $\varepsilon >0$,
$ \lim_{ \delta \downarrow 0} \limsup_{N \to \infty } \sup_{ (S,S') \in A_{\delta,T}} 
P ( |X_{S'}^{[n]} - X_S^{[n]} | + |X_{S'}^{[n+1]} - X_S^{[n]}|> \varepsilon ) = 0$,
where $A_{\delta,T}$ is the set of all pairs of stopping times $(S,S')$ such that
$0\leq S \leq S'\leq S+\delta\leq T$ a.s.,
\item[$(b)$] $\lim_{ K \uparrow \infty } \sup_n 
\mathbb{P} ( \sup_{ t \in [0, T ] } |X_t^{[n]}| + |X^{[n+1]}_t| \geq K ) = 0$.
\end{itemize}

Assertion~$(b)$ is a straightforward consequence of~\eqref{aprioripicardmckeangeneral} and Markov's inequality. To check assertion~$(a)$, notice that, for any $(S,S')\in A_{\delta,T},$ by BDG inequality, Assumption~\ref{hypmoment} and~\eqref{aprioripicardmckeangeneral},
\begin{equation}\label{aldousmckeangeneral}
\esp{\ll|X^{[n+1]}_{S'}-X^{[n+1]}_S\rr|}\leq \delta \ll(1+\esp{\underset{0\leq s\leq T}{\sup}|X^{[n+1]}_s|}\rr) + ||\sigma||_{\infty}\sqrt{\delta}\leq (\delta + \sqrt{\delta}) C_T.
\end{equation}

Then, by tightness, there exists a subsequence of $(X^{[n]},X^{[n+1]})_n$ that converges in distribution to some~$(X,Y)$ in~$D([0,T],\r^2).$ In the rest of the proof, we work on this subsequence without writing it explicitly for the sake of notation.

{\it Step~3.} In this step, we show that $X=Y$ almost surely. Note that, since we work on a subsequence, this is not obvious. It is for this part of the proof that we need to restrict our processes to a time interval of the form~$[0,T].$ It is sufficient to prove that, for a subsequence,
\begin{equation}
\label{XnXnplusunmckeangeneral}
\esp{\underset{0\leq s\leq T}{\sup}\ll|X^{[n+1]}_s-X^{[n]}_s\rr|}
\end{equation}
vanishes as~$n$ goes to infinity. Indeed, \eqref{XnXnplusunmckeangeneral} implies that, for another subsequence, $\sup_{s\leq T}|X^{[n+1]}_s-X^{[n]}_s|$ converges to zero almost surely. Then, we can apply Skorohod representation theorem (see Theorem~6.7 of \cite{billingsley_convergence_1999}) to the following sequence
$$\left(X^{[n]},X^{[n+1]}\right)_n$$
that converges in distribution in $D([0,T],\r^2)$ to $(X,Y).$ Thus we can consider, for $n\in\n,$ random variables $(\tilde X^{[n]},\tilde X^{[n+1]})$ (resp. $(\tilde X,\tilde Y)$) having the same distribution as $(X^{[n]},X^{[n+1]})$ (resp. $(X,Y)$)  for which the previous convergence is almost sure. In particular, we also know that $\sup_{s\leq T}|\tilde X^{[n+1]}_s-\tilde X^{[n]}_s|$ vanishes almost surely. As a consequence $\tilde X=\tilde Y$ almost surely, and so $X=Y$ almost surely.


Now let us prove~\eqref{XnXnplusunmckeangeneral}. Let

$$u^{[n]}(t) := \esp{\underset{0\leq s\leq t}{\sup}\ll|X^{[n+1]}_s-X^{[n]}_s\rr|}.$$

By~\eqref{aprioripicardmckeangeneral},
$$\underset{n\in\n}{\sup}~u^{[n]}(t)<\infty.$$

Let us fix some $n\in\n^*$ and consider a truncation function $r^{[n]}_t>0$ whose values will be fixed later. The same truncation argument used in Section~\ref{uniquenessmckeangeneral} allows to prove that, for all $0\leq k\leq n-1, t\leq T,$
\begin{align*}
u^{[k+1]}(t)\leq& C_T\int_0^t \ll[(1+r^{[n]}_s)u^{[k]}(s) + e^{-r^{[n]}_s}\rr]ds + 2L\sqrt{T}u^{[k]}(t)\\
\leq& C_T\int_0^t \ll[(1+r^{[n]}_s)u^{[k]}(s) + e^{-r^{[n]}_s}\rr]ds + \frac12 u^{[k]}(t).
\end{align*}
where $C_T>0$ does not depend on~$n$ thanks to~\eqref{aprioripicardmckeangeneral}. The second inequality above comes from the fact that we fix the value of $T>0$ such that $L\sqrt{T}<1/4.$

Now, introducing $S_n(t) := \sum_{k=0}^nu^{[k]}_t$ and summing the above inequality from $ k=0 $ to $k=n-1,$ we have, for all $t\leq T,$
$$S_n(t)\leq C_T+ C_T\int_0^t \ll[(1+r^{[n]}_s) S_n(s)ds + ne^{-r^{[n]}_s}]\rr)ds + \frac12 S_n(t),$$
where we have used that $ u_t^{[0]} \le C_T$ and $S_{n-1}(t)\leq S_n(t).$ 
This implies
$$S_n(t)\leq C_T+ C_T\int_0^t \ll[(1+r^{[n]}_s) S_n(s)ds + ne^{-r^{[n]}_s}\rr]ds.$$
Let $D_T := \max(\underset{k\geq 0}{\sup} \; \underset{s \le T}{\sup} |u^{[k]}_s | ,C_T,1)<\infty,$ and introduce
$$R_n(t) := \frac{S_n(t)}{(n+1)D_Te^{2}} \leq e^{-2}.$$
Consequently, for all $t\leq T,$
$$R_n(t)\leq \frac{1}{n+1}+C_T \int_0^t\ll[(1+r^{[n]}_s)R_n(s)+e^{-r^{[n]}_s}\rr]ds.$$
Finally we choose $r^{[n]}_t := -\ln R_n(t)\geq 2$ and obtain for all $ t \le T,$ 
$$
R_n(t)\leq \frac{1}{n+1}+ C_T\int_0^t\ll(2-\ln R_n(s)\rr)R_n(s)ds 
\leq\frac{1}{n+1} - C_T\int_0^t R_n(s)\ln R_n(s)ds.
$$
As before we apply Osgood's lemma. Let $M(x):=\int_x^{e^{-2}}\frac{1}{-s\ln s}ds = \ln(-\ln x) - \ln 2.$ Then
$$-M(R_n (T))+M( 1/(n+1) )\leq C_T T$$
or equivalently
$$R_n (T) \le (n+1)^{ - e^{- C_T T}} \; \mbox{ such that }
\; S_n(T)\leq C_T n^{ 1 - e^{-  C_T T} }.$$

Lemma~\ref{seriecv0} above then implies that there exists a subsequence of $(u^{[n]}_T)_n$ that converges to $0$ as $n$ goes to infinity. This proves~\eqref{XnXnplusunmckeangeneral}.

{\it Step~4.} Let us prove that a subsequence of $(\mu^{[n]})_n$ converges to some limit $\mu:t\mapsto \mu_t$ in the following sense
$$\underset{0\leq t\leq T}{\sup}W_1(\mu^{[n]}_t,\mu_t)\underset{n\rightarrow\infty}{\longrightarrow}0,$$
where $\mu_t := \mathcal{L}(X_t)$ for a.e. $t\leq T.$

We prove this point by proving that the sequence of functions $\mu^{[n]}:t\mapsto \mu^{[n]}_t = \mathcal{L}(X^{[n]}_t)\in\mathcal{P}_1(\r)$ is relatively compact, using Arzel\`a-Ascoli's theorem.

To begin with, the definition of~$W_1$ and the same computation as the one used to obtain~\eqref{aldousmckeangeneral} allows to prove that, for all $s,t\leq T,$ for all $n\in\n,$
\begin{equation}
\label{ascolimodule}
W_1(\mu^{[n]}_t,\mu^{[n]}_s)\leq \esp{\ll|X^{[n]}_t-X^{[n]}_s\rr|}\leq C \ll(|t-s| + \sqrt{|t-s|}\rr),
\end{equation}
for a constant $C>0$ independent of~$n$.

This implies that the sequence $\mu^{[n]}:t\mapsto\mu^{[n]}_t$ is equicontinuous. In addition, by~\eqref{aprioripicardmckeangeneral} we know that, for every $t\leq T,$ the set $(\mu^{[n]}_t)_n$ is tight, and whence relatively compact (in the topology of the weak convergence, but not in $\mathcal{P}_1(\r)$ a priori). Indeed, for any $\eps>0,$ considering $M_\eps := \sup_n \esp{|X^{[n]}_t|}/\eps,$ we have, for all $n,$
$$\mu^{[n]}_t(\r\backslash [-M_\eps,M_\eps]) = \pro{|X^{[n]}_t|> M_\eps}\leq \frac1{M_\eps}\esp{|X^{[n]}_t|}\leq\eps.$$

In particular, for every $t\leq T,$ we can consider a subsequence of $(\mu^{[n]}_t)_n$ that converges weakly. To prove that this convergence holds for the metric $W_1,$ we rely on the characterization~$(iii)$ of $W_1$ given in Definition~6.8, and Theorem~6.9 of \cite{villani_optimal_2008}. According to this result, the convergence of the same subsequence of $(\mu^{[n]}_t)_n$ for $W_1$ follows from~\eqref{aprioripicardmckeangeneral}, Markov's inequality, Cauchy-Schwarz's inequality and the fact that,
$$\esp{|X^{[n]}_t|\uno{|X^{[n]}_t|>R|}}\leq \frac1R ~\underset{k\in\n}{\sup}~\esp{|X^{[k]}_t|}\esp{(X^{[k]}_t)^2}^{1/2}\underset{R\rightarrow\infty}{\longrightarrow}0.$$

We can then conclude that, for all $t\leq T,$ the sequence $(\mu^{[n]}_t)_n$ is also relatively compact on~$\mathcal{P}_1(\r).$ 

Then, thanks to~\eqref{ascolimodule}, Arzel\`a-Ascoli's theorem implies that the sequence $(\mu^{[n]})_n$ is relatively compact. As a consequence, there exists a subsequence of $(\mu^{[n]})_n$ (as previously, we do not write this subsequence explicitly in the notation) that converges to some $\mu:t\mapsto\mu_t\in\mathcal{P}_1(\r)$ in the following sense
$$\underset{0\leq t\leq T}{\sup}W_1(\mu^{[n]}_t,\mu_t)\underset{n\rightarrow\infty}{\longrightarrow}0.$$


The last thing to show in this step is that $\mu_t = \mathcal{L}(X_t)$ for a.e. $t\leq T.$ By construction, $\mu_t$ is the limit of $\mu^{[n]}_t := \mathcal{L}(X^{[n]}_t).$ Recalling that $X^{[n]}$ converges to $X$ in distribution in Skorohod topology, we know that for all continuity point~$t$ of $s\mapsto \mathcal{L}(X_s)$, $\mu_t = \mathcal{L}(X_t).$
%

{\it Step~5.} Recall that, for a subsequence, $(X^{[n]},X^{[n+1]})_n$ converges to $(X,X)$ in distribution in $D([0,T],\r^2),$ and $(\mu^{[n]})_n$ (which is a sequence of deterministic and continuous functions from $\r_+$ to $\mathcal{P}_1(\r)$) converges uniformly to $\mu$ on $[0,T]$. The aim of this step is to prove that $(X^{[n]},X^{[n+1]},\mu^{[n]})_n$ converges to $(X,X,\mu)$ in distribution in $D([0,T],\r^2\times\mathcal{P}_1(\r)).$ We consider $\mu^{[n]}$ in the previous distribution even though it is deterministic, because the important point in the convergence we want to prove is that $\mu^{[n]}$ must converge w.r.t. the same sequence of time-changes as the one of $(X^{[n]},X^{[n+1]})$. In particular, it is important to have convergence in the topology of $D([0,T],\r^2\times\mathcal{P}_1(\r))$ rather than in the weaker topology $D([0,T],\r^2)\times D([0,T],\mathcal{P}_1(\r)).$

As $\mu$ is continuous, we have that, for any sequence of time-changes $(\lambda_n)_n$ the convergence
\begin{equation}
\label{cvlambdamu}
\underset{0\leq t\leq T}{\sup}W_1(\mu^{[n]}_t,\mu_{\lambda_n(t)})\underset{n\rightarrow\infty}{\longrightarrow}0.
\end{equation}

By Skorohod's representation theorem, we can assume that some representative r.v. $(\tilde X^{[n]},\tilde X^{[n+1]})$ of $(X^{[n]},X^{[n+1]})$ converges a.s. to representative r.v. $(\tilde X,\tilde X)$ of $(X,X)$ in $D([0,T],\r^2).$ This implies that, almost surely, there exists a sequence of time-changes $(\lambda_n)_n$ such that
$$\underset{0\leq t\leq T}{\sup}\ll|\tilde X^{[n]}_t - \tilde X_{\lambda_n(t)}\rr|\textrm{ and }\underset{0\leq t\leq T}{\sup}\ll|\tilde X^{[n+1]}_t - \tilde X_{\lambda_n(t)}\rr|$$
vanish as $n$ goes to infinity. So, by~\eqref{cvlambdamu}, almost surely, there exists a sequence of time-changes $(\lambda_n)_n$ such that
$$\underset{0\leq t\leq T}{\sup} d\ll[\ll(\tilde X^{[n]}_t,\tilde X^{[n+1]}_t,\mu^{[n]}]_t\rr),\ll(\tilde X_{\lambda_n(t)},\tilde X_{\lambda_n}(t),\mu_{\lambda_n(t)}\rr)\rr]\underset{n\rightarrow\infty}{\longrightarrow}0,$$
with $d[(x,y,m),(x',y',m')] = |x-x'|+|y-y'|+W_1(m,m').$ In particular, we know that the sequence $(\tilde X^{[n]},\tilde X^{[n+1]},\mu^{[n]})_n$ converges to $(\tilde X,\tilde X,\mu)$ almost surely in $D([0,T],\r^2\times\mathcal{P}_1(\r)).$ This implies that $(X^{[n]},X^{[n+1]},\mu^{[n]})_n$ converges to $(X,X,\mu)$ in distribution in $D([0,T],\r^2\times\mathcal{P}_1(\r)).$


{\it Step~6.} This step concludes the proof, showing that $X$ is solution to~\eqref{mckeangeneral}. In order to prove that $X$ is solution to~\eqref{mckeangeneral}, we use the fact that, using the notation of Definitions~II.2.6 and~II.2.16 of \cite{jacod_limit_2003}, $X^{[n+1]}$ is a semimartingale with characteristics $(B^{[n+1]}, C^{[n+1]},\nu^{[n+1]})$ given by
\begin{align*}
B^{[n+1]}_t=& \int_0^t b( X^{[n]}_s,\mu^{[n]}_s)ds,\\
C^{[n+1]}_t = & \int_0^t\sigma(X^{[n]}_s,\mu^{[n]}_s)^2ds,\\
\nu^{[n+1]}(dt,dx)=&f( X^{[n]}_t,\mu^{[n]}_t)dt \int_E\delta_{\Phi(X^{[n]}_t,\mu^{[n]}_t,u)}(dx)d\rho(u).
\end{align*}
Let us note that, above, we have chosen as truncation function $h=0,$ hence the modified second characteristics $\tilde C^{[n+1]}$ is the same as $C^{[n+1]}.$

Recall that, in {\it Step~5}, we have shown that, for a subsequence, $(X^{[n]},X^{[n+1]},\mu^{[n]})_n$ converges in distribution in $D([0,T],\r^2\times\mathcal{P}_1(\r))$ to $(X,X,\mu).$ Using once again Skorohod's representation theorem, we can consider representative r.v. for which the previous convergence is almost sure. Whence { (recalling that for $x\in D(\r_+,\r),$ the function $t\mapsto \int_0^t x_sds$ is continuous for Skorokhod topology)}, for all $g\in C_b(\r),$ the following convergences hold almost surely for the representative r.v. and hence in distribution:
\begin{align*}
&\left(X^{[n+1]},B^{[n+1]},C^{[n+1]}\right)\stackrel{\mathcal{L}}{\underset{n\rightarrow+\infty}{\longrightarrow}}\left(X,\int_0^\cdot b(X_s,\mu_s)ds,\int_0^\cdot \sigma(X_s,\mu_s)^2ds\right),\\
&\left(X^{[n+1]},\int_{[0,\cdot]\times\r}g(x)\nu^{[n+1]}(ds,dx)\right)\stackrel{\mathcal{L}}{\underset{n\rightarrow+\infty}{\longrightarrow}}\left(X,\int_0^\cdot \int_E g(\Phi(X_s,\mu_s,u))f(X_s,\mu_s)d\rho(u)ds\right),
\end{align*}
where the convergences hold respectively in the spaces $D(\r_+,\r^3)$ and $D(\r_+,\r^2).$

Then, Theorem~IX.2.4 of \cite{jacod_limit_2003} implies that $X$ is a semimartingale with characteristics $(B,C,\nu)$ given by
\begin{align*}
B_t=& \int_0^t b(X_s,\mu_s)ds,\\
C_t = & \int_0^t\sigma(X_s,\mu_s)^2ds,\\
\nu(dt,dx)=&f(X_t,\mu_t)dt \int_E\delta_{\Phi(X_t,\mu_t,u)}(dx)d\rho(u).
\end{align*}

Then, we can use the canonical representation of $X$ (see Theorem~II.2.34 of \cite{jacod_limit_2003}): $X = X_0 + B + M^c + Id * \mu^X,$ where $M^c$ is a continuous locale martingale, $\mu^X=\sum_s\uno{\Delta Y_s\neq 0}\delta_{(s,X_s)}$ is the jump measure of~$X$ (let us recall that we chose the truncation function $h=0$) and $(Id * \mu^X)_t := \int_0^t\int_\r x d\mu^X(s,x)$. By definition of the characteristics, $\langle M^c\rangle_t = C_t.$ Whence, by Theorem~II.7.1 of \cite{ikeda_stochastic_1989}, there exists a Brownian motion~$W$ such that
\begin{equation}
\label{mcmckeangeneral}
M^c_t = \int_0^t\sigma(X_s,\mu_s)dW_s.
\end{equation}

In addition, we know that $\nu$ is the compensator of $\mu^X$. We rely on Theorem~II.7.4 of \cite{ikeda_stochastic_1989}. Using the notation therein, we introduce $Z=\r_+\times E,$ $m(dz,du)=dz\rho(du)$ and
$$\theta(t,z,u):=\Phi(X_{t-},\mu_{t-},u)\uno{z\leq f(X_{t-},\mu_{t-})}.$$
According to Theorem~II.7.4 of \cite{ikeda_stochastic_1989}, there exists a Poisson measure $\pi$ on $\r_+\times\r_+\times E$ having intensity $dt\cdot dz\cdot d\rho(u)$ such that, for all $A\in\mathcal{B}(\r),$
\begin{equation*}
\mu^X([0,t]\times A)=\int_0^t\int_0^\infty\int_E\uno{\theta(s,z,u)\in A}d\pi(s,z,u).
\end{equation*}

This implies that
\begin{equation}
\label{mdmckeangeneral}
(Id * \mu^X)_t = \int_{[0,t]\times\r_+\times E}\Phi(X_{s-},\mu_{s-},u)\uno{z\leq f(X_{s-},\mu_{s-})}d\pi(s,z,u).
\end{equation}

Finally, recalling that $X = X_0+B+M^c+Id*\mu^X,$~\eqref{mcmckeangeneral} and~\eqref{mdmckeangeneral}, we have just shown that $X$ is a weak solution to~\eqref{mckeangeneral} on $[0,T]$.
\end{proof}

\subsection{Proof of Theorem~\ref{wellposedmckeangeneral}}

In Section~\ref{uniquenessmckeangeneral}, we have proven the (global) pathwise uniqueness of solutions of~\eqref{mckeangeneral}, and, in Section~\ref{existencemckeangeneral}, the existence of a weak solution of~\eqref{mckeangeneral} on $[0,T]$, with $T=1/(16L^2)$.

Then, generalizations of Yamada-Watanabe results allows to construct a strong solution on $[0,T]$: it is a consequence of Theorem~1.5 and Lemma~2.10 of \cite{kurtz_weak_2014} (see the discussion before Lemma~2.10 or Example~2.14 for more details).

More precisely, given a Brownian motion~$W,$ a Poisson random measure~$\pi$ and an initial condition~$X_0,$ there exists a strong solution $(X_t)_{0\leq t\leq T}$ defined w.r.t. these $W,\pi,X_0.$ Then, one can construct a strong solution~$(X_t)_{T\leq t\leq 2T}$ on $[T,2T]$ defined w.r.t. the Brownian motion $(W_{T+t}-W_T)_{t\geq 0},$ the Poisson measure $\pi_T$ defined by
$$\pi_T(A\times B) = \pi(\{T+x:x\in A\}\times B),$$
and the initial condition~$X_T.$ Iterating this reasoning, we can construct a strong solution of~\eqref{mckeangeneral} on $[0,kT]$ for any $k\in\n^*$, with $T=1/(16L^2)>0.$ Hence, there exists a (global) strong solution of~\eqref{mckeangeneral}. This proves Theorem~\ref{wellposedmckeangeneral}.


\section{Propagation of chaos}
\label{sectionchaosmckeangeneral}

In this section, we prove a propagation of chaos for McKean-Vlasov systems: Theorem~\ref{chaosmckeangeneral}. This property in the globally Lipschitz case has been proven in Proposition~3.1 of \cite{andreis_mckeanvlasov_2018}. Let us introduce the $N-$particle system $(X^{N,i})_{1\leq i\leq N}$

\begin{align}
dX^{N,i}_t=& b(X^{N,i}_t,\mu^N_t)dt + \sigma(X^{N,i}_t,\mu^N_t)dW^i_t + \int_{\r_+\times F^{\n^*}}\Psi(X^{N,i}_{t-},\mu^N_{t-},v^i)\uno{z\leq f(X^{N,i}_{t-},\mu^N_{t-})}d\pi^i(t,z,v)\nonumber\\
&+ \frac1N\sum_{\substack{j=1\\j\neq i}}^N\int_{\r_+\times F^{\n^*}}\Theta(X^{N,j}_{t-},X^{N,i}_{t-},\mu^N_{t-},v^j,v^i)\uno{z\leq f(X^{N,j}_{t-},\mu^N_{t-})}d\pi^j(t,z,v),\label{xnmckeangeneral}
\end{align}
with $\mu^N:=N^{-1}\sum_{j=1}^N\delta_{X^{N,j}},$ $W^i$ ($i\geq 1$) independent standard one-dimensional Brownian motions, and $\pi^i$ ($i\geq 1$) independent Poisson measures on $\r_+^2\times F^{\n^*}$ with intensity $dt\cdot dz\cdot d\nu(v),$ where $F$ is a measurable space, and $\nu$ is a $\sigma-$finite symmetric measure on~$F^{\n^*}$ (i.e. $\nu$ is invariant under finite permutations).


In the following, we assume that $b,\sigma$ and $f$ satisfy the same conditions as in Assumptions~\ref{hypmoment} and~\ref{hypnonlip}, and that $\Psi$ satisfies the same as $\Phi$ with $E = F^{\n^*}$ and $\rho = \nu.$ We also assume that $\Theta$ satisfies similar conditions: for all $x_1,x_1',x_2,x_2'\in\r,$ $m_1,m_2\in\mathcal{P}_1(\r),$
\begin{multline*}
\int_{F^{\n^*}}\int_{\r_+}|\Theta(x_1,x_1',m_1,v^1,v^2)\uno{z\leq f(x_1,m_1)} - \Theta(x_2,x_2',m_2,v^1,v^2)\uno{z\leq f(x_2,m_2)}|dzd\nu(v)\\
\leq L\ll(1+|x_1|+|x_1'|+|x_2|+|x_2'| + \int_\r |x|^{k_0}dm_1(x) + \int_\r |x|^{k_0}dm_2(x)\rr)\\
\ll(|x_1-x_2|+|x_1'-x_2'|+W_1(m_1,m_2)\rr),
\end{multline*}
and, for any $0<\lambda\leq 1,n\in\n^*,$
\begin{equation}\label{hyptheta}
\int_{F^{\n^*}}\ll|\ll(x_1 + \lambda \Theta(y_1,x_1,m,v^1,v^2)\rr)^n - x_1^n\rr|d\nu(v) f(y^1,m)\leq C\cdot \lambda\cdot n (1+|x_1|^n + |y_1|^n).
\end{equation}

In addition, we assume that
$$\underset{N\in\n^*}{\sup}\esp{e^{a|X^{N,1}_0|}}<\infty,$$
and that, for every $N\in\n^*,$ the system $(X^{N,i}_0)_{1\leq i\leq N}$ is i.i.d.

We prove that these $N-$particles systems converge as $N$ goes to infinity to the following limit system.

\begin{align}
d\bar X^{i}_t=& b(\bar X^{i}_t,\bar \mu_t)dt + \sigma(\bar X^{i}_t,\bar \mu_t)dW^i_t + \int_{\r_+\times F^{\n^*}}\Psi(\bar X^{i}_{t-},\bar\mu_{t-},v^i)\uno{z\leq f(\bar X^{i}_{t-},\bar\mu_{t-})}d\pi^i(t,z,v)\nonumber\\
&+ \int_\r\int_{F^{\n^*}}\Theta(x,\bar X^{i}_{t},\bar \mu_{t},v^1,v^2)f(x,\bar\mu_t)d\nu(v)d\bar\mu_t(x),\label{barxmckeangeneral}
\end{align}
where $\bar\mu_t = \mathcal{L}(\bar X_t).$ We assume that the variables $\bar X^i_0$ ($i\geq 1$) are i.i.d. and satisfy
$$\esp{e^{a|\bar X^i_0|}}<\infty.$$


Let us remark that the (strong) well-posedness of equation~\eqref{barxmckeangeneral} is a consequence of Theorem~\ref{wellposedmckeangeneral} for the same $\sigma,$ $\Psi = \Phi$ and for the drift
$$b(x,m) + \int_\r\int_{F^{\n^*}}\Theta(y,x,m,v^1,v^2)f(y,m)d\nu(v)dm(y).$$

One can also prove the (strong) well-posedness of equation~\eqref{xnmckeangeneral} using a similar reasoning as the one used in the proof of Theorem~\ref{wellposedmckeangeneral}. The only difference is for the {\it Step~4} of the proof of Proposition~\ref{propexistenceweaksolutionmckeangeneral}, since, for~\eqref{xnmckeangeneral} the measure $\mu^N$ is not deterministic. Instead of proving that the sequence of measures $(\mu^{[n]})_n$ constructed in the Picard scheme is relatively compact by Arzel\`a-Ascoli's theorem, we rely exclusively on the following lemma whose proof is postponed to Appendix.

\begin{lem}
\label{cvdmesureempirique}
Let $N\in\n^*,T>0,$ and $(x^k)_{1\leq k\leq N}$ and $(x_n^k)_{1\leq k\leq N}$ $(n\in\n)$ be c\`adl\`ag functions. Define
$$\mu_n(t):=N^{-1}\sum_{k=1}^N\delta_{x_n^k(t)}\textrm{ and }\mu(t):=\sum_{k=1}^N\delta_{x^k(t)}.$$

Let $\lambda_n$ ($n\in\n$) be continuous, increasing functions satisfying $\lambda_n(0)=0,$ $\lambda_n(T)=T$, and that, for any $1\leq k\leq N$,
$$\underset{0\leq t\leq T}{\sup}\ll|x_n^k(t)-x^k(\lambda_n(t))\rr|\textrm{ and }\underset{0\leq t\leq T}{\sup}\ll|t-\lambda_n(t)\rr|$$
vanish as $n$ goes to infinity. Then,
$$\underset{0\leq t\leq T}{\sup}W_1\ll(\mu_n(t),\mu(\lambda_n(t))\rr)\underset{n\rightarrow\infty}{\longrightarrow}0.$$
%
%
\end{lem}

\begin{rem}
This lemma allows to prove that, if $(x_n,y_n)_n$ converges to $(x,y)$ in $D([0,T],(\r^N)^2),$ then, the sequence $(x_n,y_n,\mu_n)_n$ converges to $(x,y,\mu)$ in $D([0,T],(\r^N)^2\times\mathcal{P}_1(\r)).$ 
\end{rem}

In the following, we assume that $(X^{N,i})_{1\leq i\leq N}$ and $(\bar X^i)_{i\geq 1}$ are strong solutions of respectively~\eqref{xnmckeangeneral} and~\eqref{barxmckeangeneral} defined w.r.t. the same Brownian motions $W^i$ ($i\geq 1$) and the same Poisson measures $\pi^i$ ($i\geq 1$) such that all the systems are defined on the same space. In addition, we assume that the following condition holds true 
\begin{equation}
\label{condinitcv0mckeangeneral}
\eps^N_0 := \esp{\ll|X^{N,1}_0 - \bar X^1_0\rr|}\underset{N\rightarrow\infty}{\longrightarrow}0.
\end{equation}

Now let us state the main result of this section: the propagation of chaos of the $N-$particle systems, that is, the convergence of the systems $(X^{N,i})_{1\leq i\leq N}$ to the i.i.d. system $(\bar X^i)_{i\geq 1}$ as $N$ goes to infinity. We comment the convergence speed in Remark~\ref{cvspeedmckeangeneral}.

\begin{thm}
\label{chaosmckeangeneral}
We have, for all $T_0>0,$
$$\esp{\underset{0\leq t\leq T_0}{\sup}\ll|X^{N,1}_t - \bar X^1_t\rr|}\underset{N\rightarrow\infty}{\longrightarrow}0.$$

Consequently, for all $k\geq 1,$ the following weak convergence holds true:
$$\mathcal{L}(X^{N,1},X^{N,2},...,X^{N,k})\underset{N\rightarrow\infty}{\longrightarrow}\mathcal{L}(\bar X^1)\otimes\mathcal{L}(\bar X^2)\otimes...\otimes\mathcal{L}(\bar X^k),$$
in the product topology of the topology of the uniform convergence on every compact set.
\end{thm}

\begin{rem}
We just state the result of Theorem~\ref{chaosmckeangeneral} for the first coordinate because both systems $(X^{N,i})_{1\leq i\leq N}$ and $(\bar X^i)_{i\geq 1}$ are exchangeable.
\end{rem}

\begin{rem}
\label{cvspeedmckeangeneral}
In the proof of Theorem~\ref{chaosmckeangeneral}, we obtain a convergence speed for
$$\esp{\underset{0\leq t\leq T_0}{\sup}\ll|X^{N,1}_t - \bar X^1_t\rr|}\underset{N\rightarrow\infty}{\longrightarrow}0$$
that depends on~$T_0.$ Indeed, if $T := 1/(16L^2),$ the formula~\eqref{eqcvspeed} below gives a convergence speed for
$$\esp{\underset{0\leq t\leq T}{\sup}\ll|X^{N,1}_t - \bar X^1_t\rr|}$$
of the form
$$S_0^N := C^1\ll(\eps^N_0 + N^{-1/2}\rr)^{C^2},$$
for some positive constants $C^1,C^2,$ where $\eps^N_0$ is given at~\eqref{condinitcv0mckeangeneral}. And, for all $k\in\n^*,$ the convergence speed $S_k^N$ of
$$\esp{\underset{k T\leq t\leq (k+1)T}{\sup}\ll|X^{N,1}_t - \bar X^1_t\rr|}$$ can be obtained inductively by
$$S_k^N = C^1\ll(S_{k-1}^N + N^{-1/2}\rr)^{C^2}$$
for the same constants $C^1,C^2$ for all $k.$
\end{rem}

Before proving Theorem~\ref{chaosmckeangeneral}, let us state a lemma about some a priori estimates of the process~$(X^{N,1}_t)_{t\geq 0}.$

\begin{lem}
\label{apriorixnmckeangeneral}
For every $N\in\n^*$, let $(X^{N,i})_{1\leq i \leq N}$ be the solution of~\eqref{xnmckeangeneral}. For any $t\geq 0,$
\begin{itemize}
\item[(i)] $\underset{N\in\n^*}{\sup}\underset{0\leq s\leq t}{\sup}\esp{e^{a|X^{N,1}_s|}}<\infty$ and $\underset{0\leq s\leq t}{\sup}\esp{e^{a|\bar X^{1}_s|}}<\infty,$ 
\item[(ii)] $\underset{N\in\n^*}{\sup}\esp{\underset{0\leq s\leq t}{\sup}|X^{N,1}_s|}<\infty$ and $\esp{\underset{0\leq s\leq t}{\sup}|\bar X^{1}_s|}<\infty.$ 
\end{itemize}
\end{lem}

%
%

\begin{proof}[Sketch of proof of Lemma~\ref{apriorixnmckeangeneral}]
We only give the main steps of the proof of the first part of the point~$(i)$ since the main part of the proof relies on the same computations as in the proof of Lemma~\ref{apriorimckeangeneral}. The only new property of the lemma is the fact that the bounds has to be uniform in~$N$. This property is easy to prove for the point~$(ii)$ of the lemma in this model because of the normalization in~$N^{-1}$ of the jump term~$\Theta$ of~\eqref{xnmckeangeneral}.

{
For the point~$(i),$ as in the proof of Lemma~\ref{apriorimckeangeneral}, we apply Ito's formula to the function $x\mapsto |x|^n$ (for $n\geq 2$) and then, we can note that, by exchangeability,
\begin{multline*}
\esp{|X^{N,1}_t|^n}\leq \esp{|X^{N,1}_0|^n} + Ct n^{n_2+1} + Cn\int_0^t \esp{|X^{N,1}_s|^n}ds\\
+ (N-1)\int_0^t \esp{\int_{F^{\n^*}}\ll|\ll|X^{N,i}_s + N^{-1}\Theta(X^{N,2}_s,X^{N,1}_s,\mu^N_s,v^2,v^1)\rr|^n - |X^{N,1}_s| ^n\rr|d\nu(v)f(X^{N,2})}, 
\end{multline*}
where the first line above is the control of the drift term, the Brownian term, and the jump term w.r.t. the heigth jump function~$\Psi$. In particular, thanks to~\eqref{hyptheta}, and since $X^{N,1}$ and $X^{N,2}$ have the same distribution, we have
$$\esp{|X^{N,1}_t|^n}\leq \esp{|X^{N,1}_0|^n} + Ct n^{n/2+1} + Cn\int_0^t \esp{|X^{N,1}_s|^n}ds.$$

This last inequality is independent of~$N$ and is the equivalent of~\eqref{xpuisn}, so the rest of the proof is the same as in Lemma~\ref{apriorimckeangeneral}.}
%
\end{proof}


\begin{proof}[Proof of Theorem~\ref{chaosmckeangeneral}]
We follow the ideas of Proposition~3.1 of \cite{andreis_mckeanvlasov_2018}. Instead of introducing an auxiliary system, we rewrite artificially the SDE~\eqref{xnmckeangeneral} as
\begin{align*}
dX^{N,i}_t=& b(X^{N,i}_t,\mu^N_t)dt + \sigma(X^{N,i}_t,\mu^N_t)dW^i_t + \int_{\r_+\times F^{\n^*}}\Psi(X^{N,i}_{t-},\mu^N_{t-},v^i)\uno{z\leq f(X^{N,i}_{t-},\mu^N_{t-})}d\pi^i(t,z,v)\\
&+ \int_\r\int_{F^{\n^*}}\Theta(x,X^{N,i}_{t},\mu^N_{t},v^1,v^2)f(x,\mu^N_t)d\nu(v)d\mu^N_t(x) + dG^N_t,
\end{align*}
where
\begin{multline*}
G^N_t= \frac1N\sum_{j=1}^N\ll[\int_{[0,t]\times \r_+\times F^{\n^*}}\Theta(X^{N,j}_{s-},X^{N,i}_{s-},\mu^N_{s-},v^j,v^i)\uno{z\leq f(X^{N,j}_{s-},\mu^N_{s-})}d\pi^j(s,z,v)\rr.\\
\ll.-\int_0^t\int_{F^{\n^*}}\Theta(X^{N,j}_s,X^{N,i}_{s},\mu^N_{s},v^1,v^2)f(X^{N,j}_s,\mu^N_s)d\nu(v)ds\rr].
\end{multline*}

Let us define
$$u^N(t) := \esp{\underset{0\leq s\leq t}{\sup}\ll|X^{N,1}_s-\bar X^1_s\rr|}.$$


Let us note $\bar\mu^N_t := N^{-1}\sum_{j=1}^N\delta_{\bar X^j}.$ By triangle inequality and since for all $m,m'\in\mathcal{P}_1(\r),W_1(m,m')\leq W_2(m,m'),$ we have for all $t\geq 0,$
$$W_1(\mu^N_t,\bar\mu_t)\leq W_1(\mu^N_t,\bar\mu^N_t) + W_1(\bar\mu^N_t,\bar\mu_t)\leq \frac1N\sum_{j=1}^N\ll|X^{N,j}_t - \bar X^j_t\rr| + W_2(\bar\mu^N_t,\bar \mu_t).$$

Besides, using Theorem~1 of \cite{fournier_rate_2015} with $d=1,$ $p=2$ and any $q>4$ (and using Lemma~\ref{apriorixnmckeangeneral}), we have
$$\esp{W_2(\bar\mu^N_t,\bar \mu_t)} \leq C\esp{|\bar X_t|^q}^{p/q} N^{-1/2}\leq C\ll(1+\esp{e^{a\bar X_t}}\rr)^{p/q} N^{-1/2}\leq C_t N^{-1/2}.$$

As a consequence, and thanks to Lemma~\ref{apriorixnmckeangeneral}, for all $t\geq 0,$
$$\esp{\underset{0\leq s\leq t}{\sup} W_1(\mu^N_s,\mu_s)} \leq u^N(t) + C_t N^{-1/2}.$$

Whence the same truncation arguments as the ones used in Sections~\ref{uniquenessmckeangeneral} and~\ref{existencemckeangeneral} allows to prove that, for all $t\leq T,$
$$u^N(t)\leq u^N(0) + C_T\int_0^t \ll[(1+r^N_s)u^N(s) + e^{-r^N_s}\rr]ds + 2L\sqrt{T}u^N(t) + \esp{\underset{0\leq s\leq T}{\sup}|G^N_s|} + C_T N^{-1/2},$$
where $C_T$ does not depend on~$N$ thanks to Lemma~\ref{apriorixnmckeangeneral}.

Now, fixing $T=1/(16L^2)$ (such that $2L\sqrt{T}\leq 1/2$), we obtain
\begin{equation}
\label{untpremckeangeneral}
u^N(t)\leq 2u^N(0) + C_T\int_0^t \ll[(1+r^N_s)u^N(s) + e^{-r^N_s}\rr]ds + 2\esp{\underset{0\leq s\leq T}{\sup}|G^N_s|}+ C_t N^{-1/2}.
\end{equation}

To control the term $G^N_s,$ we use BDG inequality,
$$\esp{\underset{0\leq t\leq T}{\sup}\ll|G^N_t\rr|^2}\leq\frac{C}{N^2}\sum_{j=1}^N\esp{\int_0^t\int_{F^{\n^*}}\Theta(X^{N,j}_s,X^{N,i}_s,\mu^N_s,v^j,v^i)^2f(X^{N,j}_s,\mu^N_s)d\nu(v)ds}\leq \frac{C_T}{N},$$
where we have used the growth condition on $\Theta$~\eqref{hyptheta} and Lemma~\ref{apriorixnmckeangeneral}.

Hence, by Cauchy-Schwarz's inequality,
$$\esp{\underset{0\leq t\leq T}{\sup}\ll|G^N_t\rr|}\leq C_T N^{-1/2}.$$

Then we can rewrite~\eqref{untpremckeangeneral} as
$$u^N(t)\leq 2u^N(0)+ C_T\int_0^t \ll[(1+r^N_s)u^N(s) + e^{-r^N_s}\rr]ds + C_TN^{-1/2}.$$

Now, let $D_T := \max(C_T,\sup_N\sup_{s\leq T} u^N(s),1)$ which is finite by Lemma~\ref{apriorixnmckeangeneral}$.(ii),$ and define
$$v^N(t):=\frac{u^N(t)}{D_Te^2}\leq e^{-2}.$$

Choosing $r^N_t := -\ln v^N(t),$ we have for all $0\leq t\leq T,$
\begin{align*}
v^N(t)\leq& 2v^N(0) + C_T\int_0^t (2-\ln v^N(s))v^N(s)ds + \frac{C_T}{\sqrt{N}}\\
\leq& 2v^N(0) -2C_T\int_0^t v^N(s)\ln v^N(s)ds + \frac{C_T}{\sqrt{N}},
\end{align*}
and Osgood's lemma allows to conclude that 
$$-M(v(T))+M(2v^N(0)+C_T N^{-1/2})\leq 2C_T T,$$
where $M(x)=\int_x^{e^{-2}}\frac{1}{-s\ln s}ds = \ln(-\ln x) - \ln 2.$ This implies that
$$\ln(-\ln(2v^N(0)+C^1_TN^{-1/2})) - \ln(-\ln v^N(t))\leq C^2_T,$$
for some constants $C^1_T,C^2_T>0,$ where we distinguish $C^1_T$ and $C^2_T$ for clarity. This implies that, for all $0\leq t\leq T,$
\begin{equation}
\label{eqcvspeed}
v^N(t)\leq \ll(2v^N(0)+C^1_T N^{-1/2}\rr)^{\exp (-C^2_T)}.
\end{equation}

Hence
$$\esp{\underset{0\leq t\leq T}{\sup}\ll|X^{N,1}_t - \bar X^1_t\rr|}\underset{N\rightarrow\infty}{\longrightarrow}0,$$
for a $T>0$ sufficiently small that does not depend on the initial conditions (recalling that we have taken $T=1/(16L^2)$). Then, iterating this reasoning on $[T,2T],$ we can prove
$$ \esp{\underset{T\leq t\leq 2T}{\sup}\ll|X^{N,1}_t - \bar X^1_t\rr|}\underset{N\rightarrow\infty}{\longrightarrow}0,$$
noticing that the "initial conditions" on~$[T,2T]$ satisfy the same condition as~\eqref{condinitcv0mckeangeneral}:
$$\eps^N_T := \esp{\ll|X^{N,1}_T - \bar X^1_T\rr|}\underset{N\rightarrow\infty}{\longrightarrow}0.$$

Finally, by induction, we can prove that for all $k\in\n^*,$
$$\esp{\underset{0\leq t\leq kT}{\sup} \ll|X^{N,1}_t - \bar X^1_t\rr|}\underset{N\rightarrow\infty}{\longrightarrow}0,$$
which proves the result.
\end{proof}

\section{Appendix}

\subsection{Osgood's lemma}

We have used many times a generalization of Gr\"onwall's lemma, which is Osgood's lemma. Let us write it explicitly for self-containedness (see e.g. Lemma 3.4 of \cite{bahouri_fourier_2011}).

\begin{lem}\label{osgood}
Let $ \varrho $ be a measurable function from $[t_0, T ] $ to $ [0, b ] , $ $\gamma $ a locally integrable function from $[t_0, T ] $ to $ \r_+,$ and $\mu $ a continuous and non-decreasing function from $ [0,  b] $ to $\r_+.$ Suppose that for all $t \in [t_0, T ] $ and for some $ c \in ]  0, b[ , $  
$$ \varrho ( t) \le c + \int_{t_0}^t \gamma ( s) \mu ( \varrho (s) ) ds .$$ 
Then, with $ M ( x) := \int_x^b \frac{ds }{\mu (s) } , $ 
$$ - M( \varrho (t) ) + M (c) \le \int_{t_0}^t \gamma (s) ds.$$ 
\end{lem}

\subsection{Proof of Lemma~\ref{seriecv0}}

{
Let us prove the result by contraposition. So let us assume that there exists no subsequence of $(u_n)_n$ that vanishes. Then, it is sufficient to prove that there exist $n_0\in\n$ and $m>0$ such that, for any $n\geq n_0,$
\begin{equation}\label{minoree}
u_n \geq m.
\end{equation}
Indeed,~\eqref{minoree} implies that $S_n$ grows at least linearly, whence $S_n/n$ cannot vanishes.

Now we prove~\eqref{minoree} by contradiction. So we assume that, for any $n\in\n,\eps>0,$ there exists some integer $\phi(n,\eps)\geq n$ such that,
$$u_{\phi(n,\eps)} \leq \eps.$$

Then, we can define $n_1 := \phi(0,1),$ and, by induction, $n_{k+1} := \phi(n_k,2^{-(k+1)}).$ So $(u_{n_k})_k$ is a subsequence of $(u_n)_n$ that vanishes.
}

\subsection{Proof of Lemma~\ref{cvdmesureempirique}}

The proof consists in noticing that, for all $t\leq T,$

$$W_1(\mu_n(t),\mu(\lambda_n(t)))\leq \frac1N\sum_{k=1}^N\ll|x_n^k(t)-x^k(\lambda_n(t))\rr|.$$

\section*{Acknowledgements}

The author would like to thank his thesis supervisors Eva L\"ocherbach and Dasha Loukianova for their support and for fruitful discussion about Osgood's lemma and the truncation arguments used in the paper.

\bibliography{biblio}

\begin{thebibliography}{18}

\bibitem[\protect\citeauthoryear{Andreis, Dai~Pra and
  Fischer}{2018}]{andreis_mckeanvlasov_2018}
\begin{barticle}[author]
\bauthor{\bsnm{Andreis},~\bfnm{Luisa}\binits{L.}},
  \bauthor{\bsnm{Dai~Pra},~\bfnm{Paolo}\binits{P.}} \AND
  \bauthor{\bsnm{Fischer},~\bfnm{Markus}\binits{M.}}
(\byear{2018}).
\btitle{{McKean}{\textendash}{Vlasov} limit for interacting systems with
  simultaneous jumps}.
\bjournal{Stochastic Analysis and Applications}
\bvolume{36}
\bpages{960--995}.
\bdoi{10.1080/07362994.2018.1486202}
\end{barticle}
\endbibitem

\bibitem[\protect\citeauthoryear{Bahouri, Chemin and
  Danchin}{2011}]{bahouri_fourier_2011}
\begin{bbook}[author]
\bauthor{\bsnm{Bahouri},~\bfnm{Hajer}\binits{H.}},
  \bauthor{\bsnm{Chemin},~\bfnm{Jean-Yves}\binits{J.-Y.}} \AND
  \bauthor{\bsnm{Danchin},~\bfnm{Rapha{\"e}l}\binits{R.}}
(\byear{2011}).
\btitle{Fourier {Analysis} and {Nonlinear} {Partial} {Differential}
  {Equations}}.
\bseries{Grundlehren der mathematischen {Wissenschaften}}
\bvolume{343}.
\bpublisher{Springer}.
\end{bbook}
\endbibitem

\bibitem[\protect\citeauthoryear{Billingsley}{1999}]{billingsley_convergence_1999}
\begin{bbook}[author]
\bauthor{\bsnm{Billingsley},~\bfnm{Patrick}\binits{P.}}
(\byear{1999}).
\btitle{Convergence of {Probability} {Measures}},
\bedition{Second} ed.
\bpublisher{Wiley Series In Probability And Statistics}.
\end{bbook}
\endbibitem

\bibitem[\protect\citeauthoryear{Carmona, Delarue and
  Lacker}{2016}]{carmona_mean_2016}
\begin{barticle}[author]
\bauthor{\bsnm{Carmona},~\bfnm{Ren{\'e}}\binits{R.}},
  \bauthor{\bsnm{Delarue},~\bfnm{Fran{\c c}ois}\binits{F.}} \AND
  \bauthor{\bsnm{Lacker},~\bfnm{Daniel}\binits{D.}}
(\byear{2016}).
\btitle{Mean field games with common noise}.
\bjournal{The Annals of Probability}
\bvolume{44}
\bpages{3740--3803}.
\bdoi{10.1214/15-AOP1060}
\end{barticle}
\endbibitem

\bibitem[\protect\citeauthoryear{Chaudru~de
  Raynal}{2020}]{chaudru_de_raynal_strong_2020}
\begin{barticle}[author]
\bauthor{\bparticle{Chaudru~de} \bsnm{Raynal},~\bfnm{P.~E.}\binits{P.~E.}}
(\byear{2020}).
\btitle{Strong well posedness of {McKean}{\textendash}{Vlasov} stochastic
  differential equations with {H{\"o}lder} drift}.
\bjournal{Stochastic Processes and their Applications}
\bvolume{130}
\bpages{79--107}.
\bdoi{10.1016/j.spa.2019.01.006}
\end{barticle}
\endbibitem

\bibitem[\protect\citeauthoryear{De~Masi
  et~al.}{2015}]{de_masi_hydrodynamic_2015}
\begin{barticle}[author]
\bauthor{\bsnm{De~Masi},~\bfnm{A.}\binits{A.}},
  \bauthor{\bsnm{Galves},~\bfnm{A.}\binits{A.}},
  \bauthor{\bsnm{L{\"o}cherbach},~\bfnm{E.}\binits{E.}} \AND
  \bauthor{\bsnm{Presutti},~\bfnm{E.}\binits{E.}}
(\byear{2015}).
\btitle{Hydrodynamic {Limit} for {Interacting} {Neurons}}.
\bjournal{Journal of Statistical Physics}
\bvolume{158}
\bpages{866--902}.
\bdoi{10.1007/s10955-014-1145-1}
\end{barticle}
\endbibitem

\bibitem[\protect\citeauthoryear{Erny, L{\"o}cherbach and
  Loukianova}{2021}]{erny_conditional_2021}
\begin{barticle}[author]
\bauthor{\bsnm{Erny},~\bfnm{Xavier}\binits{X.}},
  \bauthor{\bsnm{L{\"o}cherbach},~\bfnm{Eva}\binits{E.}} \AND
  \bauthor{\bsnm{Loukianova},~\bfnm{Dasha}\binits{D.}}
(\byear{2021}).
\btitle{Conditional propagation of chaos for mean field systems of interacting
  neurons}.
\bjournal{Electronic Journal of Probability}
\bvolume{26}
\bpages{1--25}.
\bnote{Publisher: Institute of Mathematical Statistics and Bernoulli Society}.
\bdoi{10.1214/21-EJP580}
\end{barticle}
\endbibitem

\bibitem[\protect\citeauthoryear{Fischer and
  Livieri}{2016}]{fischer_continuous_2016}
\begin{barticle}[author]
\bauthor{\bsnm{Fischer},~\bfnm{Markus}\binits{M.}} \AND
  \bauthor{\bsnm{Livieri},~\bfnm{Giulia}\binits{G.}}
(\byear{2016}).
\btitle{Continuous time mean-variance portfolio optimization through the mean
  field approach}.
\bjournal{ESAIM: Probability and Statistics}
\bvolume{20}
\bpages{30--44}.
\bnote{Publisher: EDP Sciences}.
\bdoi{10.1051/ps/2016001}
\end{barticle}
\endbibitem

\bibitem[\protect\citeauthoryear{Fournier and
  Guillin}{2015}]{fournier_rate_2015}
\begin{barticle}[author]
\bauthor{\bsnm{Fournier},~\bfnm{Nicolas}\binits{N.}} \AND
  \bauthor{\bsnm{Guillin},~\bfnm{Arnaud}\binits{A.}}
(\byear{2015}).
\btitle{On the rate of convergence in {Wasserstein} distance of the empirical
  measure}.
\bjournal{Probability Theory and Related Fields}
\bvolume{162}
\bpages{707--738}.
\bdoi{10.1007/s00440-014-0583-7}
\end{barticle}
\endbibitem

\bibitem[\protect\citeauthoryear{Fournier and
  L{\"o}cherbach}{2016}]{fournier_toy_2016}
\begin{barticle}[author]
\bauthor{\bsnm{Fournier},~\bfnm{Nicolas}\binits{N.}} \AND
  \bauthor{\bsnm{L{\"o}cherbach},~\bfnm{Eva}\binits{E.}}
(\byear{2016}).
\btitle{On a toy model of interacting neurons}.
\bjournal{Annales de l'Institut Henri Poincar{\'e} - Probabilit{\'e}s et
  Statistiques}
\bvolume{52}
\bpages{1844--1876}.
\end{barticle}
\endbibitem

\bibitem[\protect\citeauthoryear{G{\"a}rtner}{1988}]{gartner_mckean-vlasov_1988}
\begin{barticle}[author]
\bauthor{\bsnm{G{\"a}rtner},~\bfnm{J{\"u}rgen}\binits{J.}}
(\byear{1988}).
\btitle{On the {McKean}-{Vlasov} {Limit} for {Interacting} {Diffusions}}.
\bjournal{Mathematische Nachrichten}
\bvolume{137}
\bpages{197--248}.
\bnote{\_eprint:
  https://onlinelibrary.wiley.com/doi/pdf/10.1002/mana.19881370116}.
\bdoi{https://doi.org/10.1002/mana.19881370116}
\end{barticle}
\endbibitem

\bibitem[\protect\citeauthoryear{Graham}{1992}]{graham_mckean-vlasov_1992}
\begin{barticle}[author]
\bauthor{\bsnm{Graham},~\bfnm{Carl}\binits{C.}}
(\byear{1992}).
\btitle{{McKean}-{Vlasov} {Ito}-{Skorohod} equations, and nonlinear diffusions
  with discrete jump sets}.
\bjournal{Stochastic Processes and their Applications}
\bvolume{40}
\bpages{69--82}.
\bdoi{10.1016/0304-4149(92)90138-G}
\end{barticle}
\endbibitem

\bibitem[\protect\citeauthoryear{Ikeda and
  Watanabe}{1989}]{ikeda_stochastic_1989}
\begin{bbook}[author]
\bauthor{\bsnm{Ikeda},~\bfnm{Nobuyuki}\binits{N.}} \AND
  \bauthor{\bsnm{Watanabe},~\bfnm{Shinzo}\binits{S.}}
(\byear{1989}).
\btitle{Stochastic {Differential} {Equations} and {Diffusion} {Processes}},
\bedition{Second} ed.
\bpublisher{North-Holland Publishing Company}.
\end{bbook}
\endbibitem

\bibitem[\protect\citeauthoryear{Jacod and Shiryaev}{2003}]{jacod_limit_2003}
\begin{bbook}[author]
\bauthor{\bsnm{Jacod},~\bfnm{Jean}\binits{J.}} \AND
  \bauthor{\bsnm{Shiryaev},~\bfnm{Albert~N}\binits{A.~N.}}
(\byear{2003}).
\btitle{Limit {Theorems} for {Stochastic} {Processes}},
\bedition{Second} ed.
\bpublisher{Springer-Verlag BerlinHeidelberg NewYork}.
\end{bbook}
\endbibitem

\bibitem[\protect\citeauthoryear{Kurtz}{2014}]{kurtz_weak_2014}
\begin{barticle}[author]
\bauthor{\bsnm{Kurtz},~\bfnm{Thomas}\binits{T.}}
(\byear{2014}).
\btitle{Weak and strong solutions of general stochastic models}.
\bjournal{Electronic Communications in Probability}
\bvolume{19}.
\bdoi{10.1214/ECP.v19-2833}
\end{barticle}
\endbibitem

\bibitem[\protect\citeauthoryear{Lacker}{2018}]{lacker_strong_2018}
\begin{barticle}[author]
\bauthor{\bsnm{Lacker},~\bfnm{Daniel}\binits{D.}}
(\byear{2018}).
\btitle{On a strong form of propagation of chaos for {McKean}-{Vlasov}
  equations}.
\bjournal{Electronic Communications in Probability}
\bvolume{23}.
\bnote{Publisher: The Institute of Mathematical Statistics and the Bernoulli
  Society}.
\bdoi{10.1214/18-ECP150}
\bmrnumber{MR3841406}
\end{barticle}
\endbibitem

\bibitem[\protect\citeauthoryear{Mishura and
  Veretennikov}{2020}]{mishura_existence_2020}
\begin{barticle}[author]
\bauthor{\bsnm{Mishura},~\bfnm{Y.~S.}\binits{Y.~S.}} \AND
  \bauthor{\bsnm{Veretennikov},~\bfnm{A.~Y.}\binits{A.~Y.}}
(\byear{2020}).
\btitle{Existence and uniqueness theorems for solutions of
  {McKean}{\textendash}{Vlasov} stochastic equations}.
\bjournal{Theory of Probability and Mathematical Statistics}.
\bnote{Publisher: American Mathematical Society}.
\end{barticle}
\endbibitem

\bibitem[\protect\citeauthoryear{Villani}{2008}]{villani_optimal_2008}
\begin{bbook}[author]
\bauthor{\bsnm{Villani},~\bfnm{C{\'e}dric}\binits{C.}}
(\byear{2008}).
\btitle{Optimal transport, old and new}.
\bpublisher{Springer}.
\end{bbook}
\endbibitem

\end{thebibliography}
\end{document}